\documentclass[11pt]{amsart}
\usepackage{latexsym,amssymb,amsmath,youngtab}
\usepackage{youngtab}
\usepackage{epsfig}
\usepackage{amsmath,amsthm,amssymb,amscd,MnSymbol}
\usepackage[mathscr]{eucal}
\usepackage{verbatim}
\usepackage{color}
\newtheoremstyle{custom}
  {3pt}
  {3pt}
  {\slshape}
  {}
  {\bfseries}
  {.}
  { }
   {}
\theoremstyle{custom}
\newtheorem{theorem}{Theorem}[section]
\newtheorem{proposition}[theorem]{Proposition}

\newtheorem{proposition/definition}[theorem]{Proposition/Definition}
\newtheorem{lemma}[theorem]{Lemma}
\newtheorem{corollary}[theorem]{Corollary}
\newtheorem{conjecture}[theorem]{Conjecture}

\newtheorem{prop}[theorem]{Proposition}

\theoremstyle{definition}
\newtheorem{definition}[theorem]{Definition}

\newtheorem{example}[theorem]{Example}

\newtheorem{problem}[theorem]{Problem}

\theoremstyle{remark}
\newtheorem{remark}[theorem]{Remark}

\newtheoremstyle{exercise}
  {3pt}
  {6pt}
  {}
  {}
  {\bfseries}
  {:}
  { }
   {}
\theoremstyle{exercise}
\newtheorem{exercise}[theorem]{Exercise}
\newtheoremstyle{exercises}
  {3pt}
  {6pt}
  {}
  {}
  {\bfseries}
  {:}
  {\newline}
   {}

\theoremstyle{exercise}
\newtheorem{exercises}[theorem]{Exercises}

\newcommand{\exerone}[2][{}]{\begin{exercise}{#1}{#2}\end{exercise}}

\input epsf
\def\boxit#1{\vbox{\hrule height1pt\hbox{\vrule width1pt\kern3pt
  \vbox{\kern3pt#1\kern3pt}\kern3pt\vrule width1pt}\hrule height1pt}}

\def\trank{\text{rank}}

\def\bv{\bold v}
\def\z{\zeta}

\def\BC{\mathbb C}
\def\BA{\mathbb A}\def\BR{\mathbb R}
\def\BP{\mathbb P}
\def\pp#1{\mathbb P^{#1}}

\def\pp#1{{\mathbb P}^{#1}}
\def\tdim{{\rm dim}}

\def\hd{,...,}
\def\ww{\wedge}

\def\inv{{}^{-1}}

\def\cP{{\mathcal P}}

\def\cO{{\mathcal O}}
\def\CC{\mathbb C}

\def\11{\mathbf 1}
\def\PP{\mathbb P}

\def\fh{{\mathfrak h}}

\def\fsl{{\mathfrak {sl}}}

\def\fm{{\mathfrak m}}

\def\fg{{\mathfrak g}}

\def\l{\lambda}
\def\a{\alpha}

\def\o{\omega}

\def\b{\beta}

\def\s{\sigma}

\def\d{\delta}

\def\ot{{\mathord{ \otimes } }}
\def\op{{\mathord{\,\oplus }\,}}
\def\otc{{\mathord{\otimes\cdots\otimes}\;}}

\def\ra{{\mathord{\;\rightarrow\;}}}

\def\dim{{\rm dim}\;}
\def\La#1{\Lambda^{#1}}

\def\frak{\mathfrak}
\def\fsl{\frak s\frak l}

\def\op{\oplus}
\def\BA{\Bbb A}\def\BZ{\Bbb Z}

\def\ep{\epsilon}
\def\op{\oplus}

\def\ul{\underline}

\def\s{\sigma}
\def\t{\tau}

\def\a{\alpha}
\def\b{\beta}

\def\l{\lambda}

\def\FS{\mathfrak  S}

\def\ol{\overline}

\def\BP{\mathbb  P}
\def\BC{\mathbb  C}

\def\pp#1{\mathbb  P^{#1}}
\def\cC{\mathcal  C}

\def\BR{\mathbb  R}

\def\ep{\epsilon}

\def\fg{\mathfrak  g}

\def\hd{, \hdots ,}

\def\inv{{}^{-1}}

\def\La#1{\Lambda^{#1}}

\def\pp#1{\mathbb  P^{#1}}

\def\ur{\underline {\bold R}}

\def\ra{\rightarrow}

\def\tdeg{\operatorname{deg}}
\def\tdet{\operatorname{det}}

\def\tperm{\operatorname{perm}}
\def\ttrace{\operatorname{trace}}
\def\tend{\operatorname{End}}
\def\tim{\operatorname{Image}}
\def\tdim{\operatorname{dim}}
\def\tker{\operatorname{ker}}
\def\tlim{\lim}
\def\tmod{\operatorname{mod}}

\def\tmin{\operatorname{min}}

\def\thom{\operatorname{Hom}}
\def\trank{\operatorname{rank}}

\def\ww{\wedge}

\def\ctimes{\times \cdots\times}

\def\be{\begin{equation}}
\def\ene{\end{equation}}

\def\tsgn{{\rm{sgn}}}

\DeclareMathOperator{\tlog}{log}

\def\dual{{^\vee}}

\def\vp{{\bold V\bold P}}\def\vqp{{\bold V\bold Q\bold P}}
\def\vnp{{\bold V\bold N\bold P}}\def\p{{\bold P}}
\def\np{{\bold N\bold P}}

\def\G{\Gamma}

\newcommand{\Hom}{\operatorname{Hom}}
\newcommand{\tEnd}{\operatorname{End}}

\def\trank{{\mathrm {rank}}}

\def\tmult{{\rm mult}}

\def\Dual{{\mathcal Dual}}

\def\Om{\Omega}

\def\bv{\bold v}\def\bw{\bold w}

\def\tchow{{\rm chow}}\def\tfermat{{\rm fermat}}
\def\Det{{\mathcal Det}}\def\Perm{{\mathcal Perm}}\def\Pasdet{{\mathcal Pasdet}}
\def\PD{\operatorname{Pasdet}}

\begin{document}

\title{Geometric Complexity Theory: an introduction for  geometers}
\author{J.M. Landsberg}
 \begin{abstract}  
This article is survey of recent developments in, and a tutorial on,
the approach to $\p$ v. $\np$ and related questions called Geometric Complexity Theory (GCT).  It
  is written to be  accessible to graduate students. 
Numerous open questions in algebraic geometry and representation theory
relevant for GCT are presented.
\end{abstract}
\thanks{Landsberg supported by NSF grant  DMS-1006353}
\email{jml@math.tamu.edu}
\maketitle

\section{Introduction}

This is a survey of problems dealing with the separation of complexity classes that translate
to questions in algebraic geometry and representation theory. I will refer to these translations as {\it geometric complexity
theory} (GCT), although this term has been used both more broadly and more narrowly.
I do not cover topics such as the complexity of matrix multiplication (see \cite{MR2865915} for
an overview and \cite{LOsecbnd,MR3034546} for the state of the art),  matrix rigidity (see 
\cite{MR2870721,GHILrigid}), or the GCT approach to the complexity of tensors
(see \cite{MR2932001}),  although   these topics in complexity theory  have interesting algebraic geometry
and representation theory
associated to them. 
\smallskip

The basic problem (notation is explained in \S\ref{notsect} below): Let $V$ be a   vector space, let  $G\subset GL(V)$ be a reductive
group, and let $v,w\in V$. Consider the orbit closures $\ol{G\cdot [v]}, \ol{G\cdot [w]}\subset \BP V$.
Determine if $ \ol{G\cdot [v]}\subset \ol{G\cdot [w]}$.

In more detail:  First, for computer science one
is interested in asymptotic geometry, so one has a sequence of vector spaces $V_n$, and
sequences of vectors  and groups, and one wants to know if the inclusion fails for infinitely many (or even all)  $n$ greater
than some $n_0$.  Second, usually  $G_n=GL(W_n)$, where
$W_n=\BC^{f(n)}$ for some function $f(n)$ (usually    $f(n)=n^2$) and $V_n=S^nW_n$ is a space of polynomials on $ W^*_n$.  
Third, the points $v,w$ will be of a very special nature - they will (usually) be {\it characterized
by their stabilizers} (see Definition \ref{stabdef}).

The most important example will be $V=S^n\BC^{n^2}$, $G=GL_{n^2}$,  $w=\tdet_n$, the determinant,
and $v=\ell^{n-m}\tperm_m$, the {\it padded permanent} (see \S\ref{flagconj} for the definition).

 \smallskip

  The  article is part a survey of recent developments and part tutorial directed at graduate
students. The level of difficulty of the sections varies considerably  and is not monotone (for example
\S\ref{linalgdetsect} is elementary). I have placed the most emphasis on areas where there are open questions that appear
to be both tractable and interesting. The numerous open questions scattered throughout the article  are labeled by
\lq\lq {\bf Problem}\rq\rq . 
Most of the sections can be read independently of the others.

\subsection{Overview}
Section  \S\ref{gcta} serves as a detailed introduction to the rest of the paper. In it  I describe the flagship conjecture
on determinant versus permanent  and related conjectures, introduce   relevant algebraic varieties,
and establish basic information about GCT.  In \S\ref{repthsect}, I cover background from  representation
theory.
GCT has deep connections to classical algebraic geometry - a beautiful illustration of this
is how solving an old question regarding dual varieties led to
lower bounds for the flagship conjecture, which is 
  discussed in \S\ref{lowgeomsect}, along with a use of differential geometry to get
lower bounds for a conjecture of Valiant. The boundary of the   variety $\Det_n:=\ol{GL_{n^2}\cdot \tdet_n}$ is discussed in \S\ref{bndrysect}.
The classical problem of determining the symmetries of a polynomial and how it relates to the GCT program is
discussed in \S\ref{symsect}, including geometric computations of the stabilizers of the determinant and permanent
polynomials. I believe the {\it Chow variety} of polynomials that decompose into a product of linear factors
will play a central role in advancing GCT, so I discuss it in detail in \S\ref{chowsect}, including: unpublished
results of 
 Ikenmeyer and Mkrtchyan on the kernel of the Hermite-Hadamard-Howe map,    a history 
of what is called the {\it Foulkes-Howe Conjecture} (essentially due to Hadamard),    recent work
with S. Kumar related to the {\it Alon-Tarsi Conjecture},  a longstanding conjecture in combinatorics, and an exposition of Brion's proof
of an asymptotic version of the Foulkes-Howe Conjecture.  
In \S\ref{depth3sect} I translate recent results in computer science \cite{DBLP:journals/eccc/GuptaKKS13} to geometric
language - they allow for two new, completely different formulations of Valiant's conjecture $\vp\neq\vnp$,
one involving  secant varieties of the Chow variety, and another involving  secant varieties of Veronese re-embeddings
of secant varieties of Veronese varieties. 
An exposition  of
S. Kumar's results on the non-normality of $\Det_n$ and $\ol{GL_{n^2}\cdot \ell^{n-m}\tperm_n}$ is given in \S\ref{kumarpfsect}.
In \S\ref{zhanglee}, I present
unpublished results of Li and Zhang, using work of  Maulik and Pandharipande \cite{MaulPand},
 that the degree of the hypersurface of determinantal quartic surfaces
is   $640,224$. My feeling is that any near-term lower bounds for the flagship conjecture \ref{msmainconj}
will come from classical geometry and linear algebra. I discuss this perspective in  \S\ref{linalgdetsect} 
 which consists of unpublished joint work with L. Manivel and N. Ressayre. Finally \S\ref{complexapp} is an appendix
of very basic complexity theory:   the origin of $\p$ v. $\np$,   definitions regarding circuits, and Valiant's conjectures.

\subsection{Notation} \label{notsect} Throughout $V,W$ are complex vector spaces of dimensions $\bv,\bw$.  The group of invertible linear maps $W\ra W$ is denoted $GL(W)$,  and $SL(W)$ denotes
the maps with determinant one.
Since we are dealing with $GL(W)$-varieties, their ideals and coordinate
rings will be $GL(W)$-modules. The $GL(W)$-modules appearing in the tensor algebra of $W$ are indexed by partitions,
$\pi=(p_1\hd p_q)$, where if $\pi$ is a partition of $d$, 
i.e., $p_1+\cdots +p_q=d$ and $p_1\geq p_2\geq \cdots \geq p_q\geq 0$, the module $S_{\pi}W$ appears in $W^{\ot d}$ and in no other degree.
In particular the $d$-th symmetric power is $S^dV=S_{(d)}V$ and the $d$-th exterior power is $\La d V=S_{(1\hd 1)}V=:S_{(1)^d}V$.
Write $|\pi|=d$ and $\ell(\pi)=q$.
The symmetric algebra is denoted $Sym(V):=\oplus_d S^dV$.
 For a group $G$ and a $G$-module $V$ and $v\in V$, $G_v:=\{ g\in G\mid gv=v\}\subset G$ denotes its stabilizer, and
 for a subgroup $H\subset G$, $V^H:=\{ v\in V\mid hv=v\}\subset V$ denotes the $H$-invariants in $V$. 
The irreducible representations of the permutation group on $n$ elements  $\FS_n$ are also indexed by partitions, and   $[\pi]$ denotes the $\FS_n$-module
associated to $\pi$.  Repeated numbers in partitions are sometimes expressed as exponents when there is no danger of
confusion, e.g. $(3,3,1,1,1,1)=(3^2,1^4)$.

Projective space  is $\BP V= (V\backslash 0)/\BC^*$. For $v\in V$, $[v]\in \BP V$ denotes the corresponding point in projective
pace and for any subset $Z\subset \BP V$, $\hat Z\subset V$ is the corresponding cone in $V$. 
  For a variety $X\subset \BP V$,   $I(X)\subset Sym(V^*)$ denotes its ideal, $\BC[\hat X]=Sym(V^*)/I(X)$
is the  ring of regular functions on $\hat X$, which is also $\BC[X]$, the homogeneous coordinate
ring of $X$. The singular locus of $X$ is denoted  $X_{sing}$   and $X_{smooth}$ denotes its
smooth points.  More generally for an affine variety 
$Z$, $\BC[Z]$ denotes its ring of regular functions.   For $x\in X_{smooth}$, $\hat T_xX\subset V$ denotes its affine tangent space.
For a subset $Z\subset V$ or $Z\subset \BP V$, its Zariski closure is denoted $\ol{Z}$. 


For $P\in S^dV$, and $1\leq k\leq \lfloor \frac \bv 2\rfloor$,    the linear map $P_{k,d-k}: S^kV^*\ra S^{d-k}V$ 
is called  the {\it polarization} of $P$, where $P_{k,d-k}\in S^kV\ot S^{d-k}V$ is $P$ considered
as a bilinear form, see \S\ref{firsteqnssect} for more details. I write $\ol{P}$ for the complete polarization of $P$, i.e. considering $P$   as a multilinear form and $Z(P)\subset \BP V^*$ for the zero set. 

Repeated indices appearing up and down are to be summed over.

Let $T_V\subset SL(V)$ denote a torus (diagonal matrices), and I write $T=T_V$ if $V$
is understood. When $\tdim V=n$, let  $\G_n:=T\rtimes \FS_n=\{ g\in SL(V)\mid
ghg\inv \in T \forall h\in T\}$ denote  its normalizer in $SL(V)$, where
$\FS_n$ acts as permutation matrices.

For a reductive group $G$, $\Lambda^+_G$ denotes the set of finite dimensional
irreducible $G$-modules. Since I work exclusively over $\BC$,  a group is reductive
if and only if every $G$-module admits a decomposition into a direct sum of
irreducible $G$-modules. 

The set $\{ 1\hd m\}$ will be denoted $[m]$. $\tlog$ denotes $\tlog_2$.

Let $f,g: \BR \ra \BR$ be functions. Write $f= \Om(g)$ (resp. $f= O(g)$) if and only if there exists $C>0$ and $x_0$ such that
$|f(x)|\geq C|g(x)|$ (resp. $|f(x)|\leq C|g(x)|$)  for all $x\geq x_0$.   Write $f=\o(g)$  (resp. $f=o(g)$)
if and only if for all  $C>0$ there exists   $x_0$ such  that 
$|f(x)|\geq C|g(x)|$ (resp. $|f(x)|\leq C|g(x)|$)  for all $x\geq x_0$.    These definitions are used for any ordered range and domain, in particular $\BZ$.
In particular, for a function $f(n)$,   $f=\o(1)$ means $f$  goes to infinity as $n\ra \infty$.

\exerone{Show that asymptotically, for any constants $C,D,E>1$,
$$
n^C<n^{\sqrt{n}}=2^{\sqrt{n}\tlog n}< D^n < n^n<E^{n^2}.
$$
}

\subsection{Acknowledgments} 
 First I thank Massimilano Mella for inviting me to write this survey. 
I thank: my co-authors S. Kumar, L. Manivel and N. Ressayre for their permission to include our
unpublished joint work,   B. Hasset, Z. Li and L. Zhang for permission to include  their unpublished work,
G. Malod for his wonderful picture,  K. Efremenko, J. Grochow, and C. Ikenmeyer for numerous comments, the students
at a 2012 Cortona summer course on GCT and  the students in my Spring 2013 complexity theory course at Texas A\&M for their
feedback, A. Abedessalem, M. Brion and S. Kumar for help with \S 7, E. Briand for furnishing the proof
of Proposition \ref{tildehprop} and  K. Efremenko for help with \S 8.
I also thank M. Brion, P. B\"urgisser, and  R. Tange  for many useful comments on an earlier version of this article.
Finally I thank the anonymous referee for many useful comments and suggestions.

\section{Geometric Complexity Theory}\label{gcta}
\subsection{The flagship conjecture}\label{flagconj}
Let $W=\BC^{n^2}$, and let $\tdet_n\in S^nW$ denote the determinant polynomial.
Let $n>m$ and let $\tperm_m\in  S^m\BC^{m^2}$ denote the permanent. In coordinates, 
\begin{align*}
\tdet_n&=\sum_{\s\in \FS_n}\tsgn(\s) x^1_{\s(1)}\cdots x^n_{\s(n)}\\
\tperm_m&=\sum_{\s\in \FS_m} y^1_{\s(1)}\cdots y^m_{\s(m)}.
\end{align*}
Let $\ell$ be a linear coordinate on $\BC^1$ and consider any linear inclusion
$\BC^1\oplus \BC^{m^2}\ra W$, so in particular $\ell^{n-m}\tperm_m\in S^nW$.
Let 
$$\Det_n:=\ol{GL(W)\cdot [\tdet_n]}
$$ 
and let 
$$\Perm^m_n:=\ol{GL(W)\cdot [\ell^{n-m}\tperm_m]}.
$$
\begin{conjecture}\label{msmainconj}(Mulmuley-Sohoni \cite{MS1}) Let $n=m^c$ for any constant
$c$. Then for all sufficiently large $n$, 
$$
\Perm^m_n\not\subset \Det_n.
$$
\end{conjecture}

 While this flagship conjecture   appears to be out of reach, I hope
to convince the reader that there are many interesting intermediate problems that are tractable and that these questions   have deep connections to
geometry, representation theory, combinatorics,  and other areas of mathematics. 

\smallskip

It is convenient to introduce the following notation:
For a homogeneous polynomial $P$ of degree $m$, write $\ol{dc}(P)$ for the smallest $n$ such that $[\ell^{n-m}P]\in \Det_n$, called
the {\it border determinental complexity} of $P$.
Define $dc(P)$ to be the smallest $n$ such that $\ell^{n-m}P\in \tend(W)\cdot \tdet_n$, so $\ol{dc}(P)\leq dc(P)$.
Conjecture \ref{msmainconj} can be restated that $\ol{dc}(\tperm_m)$ grows faster than any polynomial in $m$. 
For example, $\ol{dc}(\tperm_2)=2$, and it is known (respectively \cite{MR3048194} 
and   \cite{Gre11})  that $5\leq \ol{dc}(\tperm_3)\leq dc(\tperm_3) \leq 7$.
The known general lower bound is
\begin{theorem}\cite{MR3048194}\label{LMRbnd} $\ol{dc}(\tperm_m)\geq \frac{m^2}2$.
\end{theorem} 
See \S\ref{MR3048194sect} for
a discussion.

Conjecture \ref{msmainconj} is a stronger version of a conjecture of L. Valiant \cite{vali:79-3} that
$dc(\tperm_m)$ grows faster than any polynomial in $m$.  The best lower bound for $dc(\tperm_m)$ is 
\begin{theorem}\cite{MR2126826}  $dc(\tperm_m)\geq \frac {m^2}2$.
\end{theorem}
See \S\ref{MRsect} for a discussion.

\begin{problem}Determine $\ol{dc}(\tperm_3)$.
\end{problem}

\smallskip

  If you were to
prove either Valiant's conjecture or Conjecture \ref{msmainconj},
  it would be by far the most significant result since the dawn of complexity theory.  Proving Conjecture \ref{msmainconj} for $c=3$
 would already be a huge accomplishment.
 If you disprove Valiant's  conjecture plus 
(1) the
projections from $\tdet_n$  to  $\tperm_m$ use only rational constants of polynomial
bit-length, and (2) the projection  (for some
$n=m^k$) is computable by a polynomial time algorithm, then 
 you can   claim the Clay prize for showing $\p= \np$.

\smallskip

A geometer's first reaction to Conjecture \ref{msmainconj} might be: \lq\lq well, the determinant is wonderful,
it has a nice geometric description, but what about this permanent? 
It is not so wonderful at first sight \rq\rq .

In fact that was my first reaction. 
If you had this reaction,
you probably think of the determinant,  not in terms of its formula, but, letting $A,B=\BC^n$,   as the unique point in $\BP S^n(A\ot B)$ invariant
under $SL(A)\times SL(B)$, i.e., a point in the trivial $SL(A)\times SL(B)$-module
$\La n A\ot \La n B\subset S^n(A\ot B)$.
If you think this way, then   consider, instead of the permanent,
the four factor  {\it Pascal Determinant} (also called the {\it combinatorial determinant}):   
and let  $A_j=\BC^m$ for $j=1\hd 4$. The $4$-factor  Pascal determinant $\PD_{4,m}$ spans 
  the unique  trivial $SL(A_1)\ctimes SL(A_4)$-module in $S^m(A_1\otc A_4)$,
 namely $\La m A_1\otc \La mA_4$. 
 Assume $n>m^4$,  choose a linear embedding $\BC\op A_1\otc A_4\subset W$,  and
 define  $\Pasdet^m_n:=\ol{GL(W)\cdot [\ell^{n-m}\PD_{4,m}]}$. Then, a consequence of
 an observation of  
 Gurvits \cite{Gurvits} is that 
 Conjecture \ref{msmainconj} is equivalent to:

\begin{conjecture}   Let  $n=m^c$ for some constant
$c$. Then for all sufficiently large $n$, 
$$
\Pasdet^m_n\not\subset \Det_n.
$$
\end{conjecture}

That being said, I have since changed my perspective and have come around to admiring the beauty
of the permanent as well. In Remark \ref{permbeauty} we will see it is the \lq\lq next best\rq\rq\ polynomial in
$S^n(\BC^n\ot \BC^n)$ after the determinant.   

There are many similarities between the permanent and
the $4$-factor Pascal determinant. Two examples:  for both the dimension
 of the ambient space is roughly the dimension of the symmetry group $G_P$
raised to the fourth power (in contrast to the determinant where the dimension is the square
of the dimension of the symmetry group), and in both cases the tangent space $T_P(GL(W)\cdot P)$ is a reducible
  module (for the determinant it is irreducible).
  
 \begin{remark} For all  even $k$ one can define the $k$-factor Pascal determinant
as a point in  the unique    trivial $SL(A_1)\ctimes SL(A_k)$-module in $S^m(A_1\otc A_k)$,
 namely $\La m A_1\otc \La mA_k$. When $k=2$ this is just the usual determinant.
 \end{remark}

\exerone{For $P=\tdet_n$, $\tperm_m$, and $\PD_{4,m}$, determine the structure of 
$T_P(GL(W)\cdot P)$ as a $G_P$-module. Hint: for any orbit, $G\cdot v=G/H$, one has
$T_vG/H\simeq \fg/\fh$ as an $\fh$-module.}

\subsection{Relevant algebraic varieties}\label{relevantsect}
Two important varieties for our study will be the {\it Veronese variety} $v_n(\BP W)\subset \BP S^nW$ and a certain 
{\it Chow variety} $Ch_n(W)\subset \BP S^nW$. These are defined as
\begin{align}
v_n(\BP W)&=\{ [z]\in  \BP S^nW \mid z=w^n {\rm \ for \ some\ } w\in W\}\\
Ch_n(W)&=\{ [z]\in  \BP S^nW \mid z=w_1\cdots  w_n {\rm \ for \ some\ } w_j\in W\}.
\end{align}
Note that the first variety is a subvariety of the second, and if we consider the Segre variety
$$
Seg(\BP W \ctimes \BP W ):=\{ [T]\in \BP (W^{\ot n})\mid T=w_1\otc w_n {\rm\  for \ some \ } w_j\in W\} \subset \BP(W^{\ot n}),
$$
 then 
$v_n(\BP W)=Seg(\BP W \ctimes \BP W )\cap \BP(S^nW)$ and
$Ch_n(W)=proj_{\BP S^dW^c}(Seg(\BP W \ctimes \BP W ))$, where
$S^dW^c\subset W^{\ot n}$ is the $GL(W)$-complement to $S^dW$, and $proj_L$ denotes
 linear projection from the linear space $L$. (Here I am respectively considering $S^nW$
 as a subspace and as a quotient of $W^{\ot n}$.)  The Veronese is homogeneous, so in particular
its ideal and coordinate ring are well understood. The Chow variety is
an orbit closure (when $n\leq\bw$).  Determining information about its ideal is a topic of current research, and has surprising
connections to different areas of mathematics, including a longstanding conjecture in combinatorics, see \S\ref{chowcombin}. There is
a natural map
$h_{d,n}:S^d(S^nW^*)\ra S^n(S^dW^*)$, dating back to Hermite and Hadamard,  such that $I_d(Ch_n(W))=\tker (h_{d,n})$,
see \S\ref{chowsect}.

The Chow variety is a good testing ground  for GCT, so it is discussed in detail in 
\S\ref{chowsect}. In particular, the coordinate rings of the Chow variety, its normalization,
and the orbit $GL(W)\cdot (x_1\hd x_n)$ are compared. Since $Ch_n(W)\subset \Det_n$, we can get
some information about the coordinate ring of $\Det_n$ from the coordinate ring of $Ch_n(W)$.

\medskip

We will often construct auxiliary varieties from our original varieties.
Let $X\subset \BP V$ be a variety, which we assume to be irreducible and reduced. 

Define the {\it dual variety}    of $X$:
$$
X\dual : =\ol{\{ H\in \BP V^*\mid \exists x\in X_{smooth}, \ \BP \hat T_xX\subseteq H\} }\subset \BP V^*.
$$
In the special case $V=S^nW^*$ and $X=v_n(\BP W^*)$ is the Veronese variety, then
the hypersurface 
$v_n(\BP W^*)\dual\subset \BP S^nW$ may be identified with  the variety  of hypersurfaces of degree $n$ in $\BP W^*$ that are singular.
To see this, for a hypersurface $Z(P)\subset \BP W^*$ (the zero set of the polynomial $P$), $[x]\in Z(P)_{sing}$ if and only if $\ol{P}(x^{n-1}y)=0$ for all $y\in W^*$.
But $\hat T_{[x^n]}v_n(\BP W^*)=\{ x^{n-1}y\mid y\in W^*\}$. See \cite[\S 8.2.1]{MR2865915} for more details.

The set 
$$
\s_r^0(X):=  \bigcup_{x_1\hd x_r\in X} \langle x_1\hd x_r\rangle  \subset \BP V,
$$
where $\langle x_1\hd x_r\rangle$ denotes the (projective) linear span of the points $x_1\hd x_r$, 
is called the set of points of {\it $X$-rank} at most $r$. The variety $\s_r(X):=\ol{\s_r^0(X)}$ is called the
 $r$-th  {\it secant variety} of $X$ (or   {\it the variety of secant $\pp{r-1}$'s to $X$}).
Assume  $X$ is not contained in a hyperplane.   Given $z\in \BP V$, define the {\it $X$-border rank} of $z$ to be the smallest $r$ such that $z\in \s_r(X)$,
and one writes $\ur_X(z)=r$. Similarly, if $z$ has $X$-rank $r$, one writes $\bold R_X(z)=r$. 

When $X=v_n(\BP W)$, the $v_n(\BP W)$-rank is called the {\it Waring rank} (or {\it symmetric tensor rank})
 and the Waring rank and border rank of a polynomial
are first measures of its complexity. One writes  $\bold R_S=\bold R_{v_n(\BP W)}$ and $\ur_S =\ur_{v_n(\BP W)}$. 
We call the $Ch_n(W)$-rank the {\it Chow rank}. The Chow rank is an  important measure
of complexity, it is related  to the size of  the smallest {\it homogeneous depth $3$ circuit} (sometimes called a
homogeneous  $\Sigma\Pi\Sigma$ circuit)
 that can compute a   polynomial,  and even more importantly,  as the smallest depth $3$ circuit that
 can compute a padded polynomial, see \S\ref{depth3sect}.

 \subsection{First equations\label{firsteqnssect}}
Equations for the secant varieties of Chow varieties are mostly unknown, and even for the Veronese  very little is known.
One class of equations is  obtained from the so-called {\it flattenings} or {\it catalecticants}, which 
date back to Sylvester:
for $P\in S^dV$, and $1\leq k\leq \lfloor \frac \bv 2\rfloor$,  consider the linear map $P_{k,d-k}: S^kV^*\ra S^{d-k}V$,
obtained from the {\it polarization} of $P$, where, from a tensorial point of view,  $P_{k,d-k}\in S^kV\ot S^{d-k}V$ is $P$ considered
as a bilinear form on $S^kV^*\times S^{d-k}V^*$. The image of $P_{k,d-k}$,
considered as a map $S^kV^*\ra S^{d-k}V$,  is the space of all $k$-th order partial derivatives of $P$, and
is studied frequently in the computer science literature under the name the {\it method
of partial derivatives} (see, e.g., \cite{MR2901512} and the references therein). To see this description of the image, note
that $S^kV^*$ may be identified with the space of $k$-th order constant coefficient homogeneous
degree $k$ differential operators on $S^nV$. In bases, if $x^1\hd x^{\bv}$ is a basis of $V$,
then $\frac{\partial}{\partial x^1}\hd \frac{\partial}{\partial x^{\bv}}$ is a basis of $V^*$.
The kernel and image of $P_{k,n-k}$ is
often easy to compute, 
  or at least estimate.

If $[P]\in v_d(\BP V)$, the rank
of $P_{k,d-k}$ is one, so the size $(r+1)$-minors of $P_{k,d-k}$ furnish some equations in $I_{r+1}(\s_r(v_d(\BP V)))$.
The only other equations I am aware of come from {\it Young flattenings}, see \cite{MR3081636,ELS} for a discussion of the
Young flattenings and the state of the art. 
If $P\in Ch_d(V)$, then the rank of $P_{k,d-k}$ is $\binom dk$, so the size $r\binom dk +1$ minors furnish some
equations for $\s_r(Ch_d(V))$. 

 For $P\in S^dV$, the Young flattening, $P_{k,d-k[\ell]}: S^kV^*\ot S^{\ell}V\ra S^{d-k+\ell}V$ obtained
 by tensoring $P_{k,d-k}$ with the identity map $Id_{S^{\ell}V}: S^{\ell}V\ra S^{\ell}V$, and projecting (symmetrizing)
 the image in $S^{d-k}V\ot S^{\ell}V$ to $S^{d-k+\ell}V$, goes under the name \lq\lq method of shifted partial
 derivatives\rq\rq\ in the computer science literature. It is the main tool for proving the results discussed in \S\ref{depthred}.
 It's skew cousin led to the current best lower bound for the border rank of matrix multiplication in \cite{LOsecbnd}.

\subsection{Problems regarding secant varieties related to  Valiant's conjectures}\label{valprobs}

\begin{problem} Find equations in the ideal of $\s_r(Ch_n(W))$. This would enable one to prove
lower complexity bounds for depth $3$ circuits.
\end{problem}

The motivation comes from: 

\begin{conjecture}\label{chowvnp}  For all but a finite number of $m$, for all $r,n$ with $rn=2^{\sqrt{m}\tlog(m) \o(1)}$,
\be\label{chea}
[\ell^{n-m}\tperm_m]\not\in \s_r(Ch_{n}(\BC^{m^2+1})).
\ene    
\end{conjecture}

As explained in \S\ref{depthred},  Conjecture \ref{chowvnp} would imply Valiant's conjecture
that $\vp\neq\vnp$. (Valiant's conjecture is explained in  \S\ref{complexapp}.) 

\smallskip

Another variety of interest is $\s_{\rho}(v_{\d}(\s_r(v_n(\BP W))))\subset \BP S^{\d n}W$.
If $\tdim W=r\rho$, and $W$ has basis $x_{is}$, $1\leq i\leq r$, $1\leq s\leq\rho$,
 this variety is the $GL(W)$-orbit closure of the polynomial
 $\sum_{s=1}^{\rho}(x_{1s}^n+\cdots + x_{rs}^n)^{\d}$.

\begin{problem} Find equations in the ideal of $\s_{\rho}(v_{\d}(\s_r(v_n(\BP W))))$. 
\end{problem}

Such equations would  enable one to prove
lower complexity bounds for the  $\Sigma\Lambda\Sigma\Lambda\Sigma$ circuits defined in \S\ref{depth3sect}.
The motivation comes from: 

\begin{conjecture}\label{slsvnp} For all but a finite number of $m$, 
for all   $\d\simeq \sqrt{m}$ and all $r,\rho$ with  $r\rho=2^{\sqrt{m} log(m) \omega(1)}$,   
\be\label{cheb}
[\tperm_m]\not\in \s_{\rho}(v_{\d}(\s_r(v_{\frac m\d}(\pp{m^2-1})))).
\ene
\end{conjecture}

As explained in \S\ref{depthred} Conjecture \ref{slsvnp} would also imply Valiant's conjecture
that $\vp\neq\vnp$.


 Note that   although the variety appearing in  \eqref{cheb} is more
complicated than the one appearing in  \eqref{chea}, we do not have to deal with  {\it cones} and  {\it padding}, which I discuss next.

\subsection{Cones and padding}  
The inclusion $\BC^{m^2+1}\subset \BC^{n^2}$, indicates we should   consider the variety of cones, or {\it subspace variety}
$$
Sub_k(S^nW):=\{ [P]\in \BP S^nW \mid \exists U^k\subset W, P\in S^nU\}, 
$$
and   the $\ell^{n-m}$ factor in both \eqref{chea} and Conjecture \ref{msmainconj}  indicates   we should   consider the {\it variety of padded polynomials}
$$
Pad_t(S^nW):= \{ [P]\in \BP S^nW\mid P=\ell^tQ {\rm \ for \ some \ }\ell\in W, Q\in S^{n-t}W\}.
$$

The ideal of $Sub_k(S^nW)$ in degree $d$ consists of the isotypic components of all $S_{\pi}W^*$ with
$\ell(\pi)>k$, see, e.g. \cite[\S 7.1]{MR2865915}. The ideal is generated in degree $k+1$ by
the   minors of flattenings \cite{LWsecseg}. The ideal of $Pad_t(S^nW)$ is not known completely. We do know: 

\begin{theorem} \cite{KLpoly}  For all $d$,  
$I_d(Pad_{t}(S^nW^*))$ contains the isotypic component of $S_{\pi}W$ in $S^d(S^nW)$ for all
$\pi=(p_1\hd p_d)$ (so $|\pi|=nd$)  with $p_1< dt$.  It does not contain a  copy of any $S_{\pi}W$ where $p_1\geq \tmin\{d(n-1),dn-(n-t)\}$.
\end{theorem}

Although we know for dimension reasons that $Pad_{n-m}(Sub_{m^2+1}(S^m\BC^{n^2}))\not\subset \Det_n$ 
asymptotically when $n=m^c$ by counting dimensions, it would be useful to have a proof using equations.

\subsection{GCT useful modules}
One could break down the problem of  separating the determinant from the padded permanent into three steps: 
separating the determinant from a generic cone, separating the determinant from a cone over a padded polynomial,
and finally separating the determinant from the cone over the padded permanent.
That is, to separate $\Det_n$ from $\Perm^m_n$,  we should not just look for modules in the ideal of $\Det_n$,  but modules in the ideal that are not
in the ideal of  $Sub_k(S^nW)$ or $Pad_{n-m}(S^n\BC^{m^2+1})$. 

\begin{definition} A $GL(W)$-module module $M$ such that $M\subset I(\Det_n)$ and  
  $M\not\subset I(Sub_k(S^nW))$ and not  known to be in the ideal of  $Pad_{n-m}(S^n\BC^{m^2+1})$, 
is called {\it $(n,m)$-GCT useful}.
\end{definition}

More precisely, one should speak of modules that are, e.g. \lq\lq April 2013 GCT useful\rq\rq , since what is known
will change over time, but I ignore this in the notation. 
To summarize:

\begin{theorem}\cite{KLpoly}
Necessary conditions for a module  $S_{\pi}W$ with $|\pi|=dn$  to be $(n,m)$-GCT useful are 
\begin{enumerate}
\item   $\ell(\pi)\leq m+1$ and  
\item  $p_1\geq d(n-m)$.
\end{enumerate}
\end{theorem} 

\begin{problem} Find a $(5,3)$-GCT useful module. 
\end{problem}

\subsection{The   program of \cite{MS1}}
The algebraic Peter-Weyl theorem (see \S\ref{algPW}) implies  that 
for  a reductive algebraic group $G$ and  a subgroup $H$, that the   ring of regular functions on $G/H$, denoted $\BC[G/H]$,   as $G$-module
is simply
$$
\BC[G/H]=
\bigoplus_{\l\in \Lambda^+_G} V_{\l}\ot (V_{\l}^*)^H.
$$
Here $\Lambda^+_G$ indexes the irreducible $G$-modules, $V_{\l}$ is the irreducible module associated to $\l$,
and for a $G$-module $W$, $W^H:=\{ w\in W\mid h\cdot w=w\forall h\in H\}$ denotes the subspace of $H$-invariants. Here $G$ acts on the $V_{\l}$ and
$(V_{\l}^*)^H$ is just a vector space whose dimension records the multiplicity of $V_{\l}$ in $\BC[G/H]$. 

Let $v\in V$ and consider the homogeneous space $G\cdot v=G/G_v \subset V$. Then  
 there is an injection $ \BC[\ol{G\cdot v}] \ra \BC[G/G_v]$ by restriction of functions. Thus  if we can find a module $V_{\l}$  that occurs in $Sym(V^*)$ that does
not occur in $\BC[G/G_v]$, the isotypic component of $V_{\l}$ in $Sym(V^*)$ must be in the ideal
of $\ol{G\cdot v}$. More generally, if the multiplicity of $V_{\l}$ in $Sym(V^*)$ is higher than
its multiplicity in $\BC[G/G_v]$, at least some copy of it must occur in $I(\ol{G\cdot v})$.

\begin{definition} Let $v\in V$ as above.  An irreducible $G$-module $V_{\l}$ is an {\it orbit occurrence obstruction} for
$\ol{G\cdot v}$ if 
$V_{\l}\subset Sym(V^*)$ and $(V_{\l})^{*G_v}=0$. The module $V_{\l}$ is an  {\it orbit representation-theoretic obstruction}
if $\tmult(V_{\l},Sym(V^*))>\tdim (V_{\l})^{*G_v}$. 
More generally, an irreducible $G$-module $V_{\l}$ is an {\it   occurrence obstruction} if 
$V_{\l}\subset Sym(V^*)$ and $ V_{\l}\not\in \BC[\ol{G\cdot v}]$. The module $V_{\l}$ is a   {\it  representation-theoretic obstruction}
if $\tmult(V_{\l},Sym(V^*))>\tmult  (V_{\l}\,\BC[\ol{G\cdot v}])$. 
\end{definition}

Note   the implications:
$M$ is an orbit occurrence obstruction implies $M$ is an orbit representation-theoretic obstruction implies
$M$ is a representation-theoretic obstruction and $M$ is an occurrence obstruction implies $M$ is a representation-theoretic
obstruction.

To summarize: The isotypic component of an occurrence obstruction in $Sym(V^*)$
is in the ideal of $\ol{G\cdot v}$, and at least some copy of a representation-theoretic obstruction
must be in the ideal of $\ol{G\cdot v}$.

 The program initiated in \cite{MS1} and continued in \cite{MS2,MR2927658} and other preprints, was to find such obstructions
 via representation theory, perhaps using   canonical bases for nonstandard quantum groups, see especially \cite{gctVII,gctVIII}. 

\begin{definition}\label{stabdef} $P\in S^dV$ is {\it characterized by}  $G_P$ if any $Q\in S^dV$ with $G_Q\supseteq G_P$ is of the
form $Q=cP$ for some constant $c$.
\end{definition}
  In our situations (where $G_P$ is reductive), the orbit closure 
of a polynomial characterized by its symmetry group  is essentially
determined by multiplicity data, which makes one more optimistic for representation-theoretic, or even occurrence obstructions.

 In the negative direction, 
C. Ikenmeyer \cite[Conj. 8.1.2]{ikenthesis} made numerous computations that lead him to 
 conjecture  that all   modules that occur in $Sym(S^nW^*)$   when $n$ and the partitions
are both even, also occur in $\BC[GL(W)\cdot \tdet_n]$.

\subsection{The boundary of $\Det_n$}\label{bndrysect}
When Conjecture \ref{msmainconj}  was first proposed, it was not known if   the inclusion  $End(W)\cdot det_n\subset  \Det_n$
was proper. In \S\ref{bndrysect},
I describe an explicit component of $\partial\Det_n$ (found in \cite{MR3048194})  that is not contained in $End(W)\cdot det_n$. 
Determining the components of the boundary should be very useful for GCT. It also relates to
a classical question in linear algebra: determine the unextendable linear spaces on $\{ \tdet_n=0\}$.

\subsection{Bad news}

Hartog's theorem states that a holomorphic function defined off of a codimension two subset of a complex manifold
extends to be defined on the complex manifold. Its analog in algebraic geometry, for say affine varieties,  is true 
in the sense that a function defined off of a codimension two subvariety of an affine variety $Z$ extends to be defined
on all of $Z$  as long as 
the affine  variety $Z$ is
{\it normal} (see \S\ref{normalizationsect} for the definition of normal). When studying a normal orbit closure, the only difference between
$\BC[G\cdot v]$ and $\BC[\ol{G\cdot v}]$ comes from functions having poles along a component of the boundary. With non-normal
varieties the situation is far subtler. The following theorem and its proof are discussed in \S\ref{kumarpfsect}.

\begin{theorem}[Kumar \cite{MR3093509}]
$\Det_n$ is not normal for $n\geq 3$.
$\Perm^m_n$ is not normal for   $n>2m$.
\end{theorem}

\begin{remark} In \cite{MR2590845} an algorithm is described that in principle can distinguish when one orbit closure
is contained in another.
\end{remark}

\section{Representation theory}\label{repthsect}

\subsection{The algebraic Peter-Weyl theorem}\label{algPW}  Let $G$ be a
reductive  algebraic group and $V$ a $G$-module. Given $\a\in V^*$ and
$v\in V$ we get an algebraic function
\begin{align*}f_{\a\ot v}: G&\ra \BC\\
g&\mapsto \a(gv). 
\end{align*}
  Note this is linear in $V$ and $V^*$ (e.g. $f_{(\a_1 + \a_2)\ot v}=f_{\a_1\ot v}+f_{\a_2\ot v}$ etc..), 
  and commutes with the action of $G$, 
so we obtain an injective  $G$-module map  $V^*\ot V\ra \BC[G]$.

\exerone{Show the map $V^*\ot V\ra \BC[G]$ is indeed injective.}

The linearity shows that it is sufficient to consider irreducible modules to avoid redundancies.
We have shown: $\BC[G]\supseteq \oplus_{\l\in \Lambda^+_G}V_{\l}^*\ot V_{\l}$.

\begin{theorem} [Algebraic Peter-Weyl]  (see, e.g \cite[Ch. 7, \S3.1.1]{MR2265844})  As a left-right $G\times G$ module,  $\BC[G]= \oplus_{\l\in \Lambda^+_G}V_{\l}^*\ot V_{\l}$.
\end{theorem}
The $G\times G$ module structure is given by $(g_1,g_2)f(g)=f(g_1gg_2)$. For the proof of the equality (which is not difficult), see \cite[p 160]{MR2265844}.

We will need the following Corollary: 
\begin{corollary} \label{orbitcor} Let $H\subset G$ be a   closed  subgroup. Then, as a $G$-module, 
$$
\BC[G/H]=\BC[G]^H=\oplus_{\l \in \Lambda_G^+} V_{\l}\ot (V_{\l}^*)^H = \oplus_{\l \in \Lambda_G^+} V_{\l}^{\oplus \tdim (V_{\l}^*)^H}.
$$
\end{corollary}

\subsection{Representations of $GL(V)$}\label{glwreps}  
The irreducible representations of $GL(V)$ are indexed by sequences $\pi=(p_1\hd p_l)$ of non-increasing integers with $l\leq \tdim V$. 
Those that occur in $V^{\ot d}$ are partitions of $d$, and we write $|\pi|=d$ and   $S_{\pi}V$ for the module.
$V^{\ot d}$ is also an $\FS_d$ module, and  the groups $GL(V)$ and $\FS_d$ are the commutants of each other in $V^{\ot d}$ which implies
the famous Schur-Weyl duality that $V^{\ot d}=\oplus_{|\pi|=d,\ell(\pi)\leq \bv} [\pi]\ot S_{\pi}V$ as a
$(  \FS_d \times  GL(V))$-module, where $[\pi]$ is the irreducible
$\FS_d$-module associated to $\pi$. Repeated numbers in partitions are sometimes expressed as exponents when there is no danger of
confusion, e.g. $(3,3,1,1,1,1)=(3^2,1^4)$.
For example, $S_{(d)}V=S^dV$ and $S_{(1^d)}V=\La d V$. The modules  $S_{s^{\bv}}V=(\La {\bv}V)^{\ot s}$ 
  are exactly the   $SL(V)$-trivial modules.
The module $S_{(22)}V$ is the home of the Riemann curvature tensor in Riemannian geometry.
 See any of \cite[Chap. 6]{MR2865915}, \cite[Chap 6]{FH} or \cite[Chap. 9]{MR2265844} for more details on the
representations of $GL(V)$, Schur-Weyl duality,  and what follows.

Assuming $\bv,\bw$ are sufficiently large, we may write:
\begin{align}
S_{\pi}(V\op W)&= \bigoplus_{|\mu|+|\nu|=|\pi|}(S_{\mu}V\ot S_{\nu}W)^{\op c^{\pi}_{\mu\nu}}\\
S_{\pi}(V\ot W)&=\bigoplus_{|\mu|=|\nu|=|\pi|}(S_{\mu}V\ot S_{\nu}W)^{\op k_{\pi \mu\nu}}
\end{align}
for some non-negative integers $ c^{\pi}_{\mu\nu}, k_{\pi \mu\nu}$.
On the left  hand side one respectively has $GL(V\op W)$ and $GL(V\ot W)$ modules and on the
right  hand side $GL(V)\times GL(W)$-modules. The constants $c^{\pi}_{\mu\nu}$ are called {\it Littlewood-Richardson
coefficients} and the $k_{\pi,\mu,\nu}$ are called {\it Kronecker coefficients}. They are independent of the dimensions
of the vector spaces as long as $\bv,\bw$ are larger than the lengths of the partitions.
They also (via Schur-Weyl duality) admit descriptions in terms of the symmetric group:
\begin{align}
c^{\pi}_{\mu\nu}&=\tdim(\thom_{\FS_{|\mu|}\times \FS_{|\nu|}}([\pi],[\mu]\ot [\nu])\\
k_{\pi\mu\nu}&=\tdim([\pi]\ot[\mu]\ot [\nu])^{\FS_d}
\end{align}
where in the first line $|\pi|=|\mu|+|\nu|$ and in the second line $|\pi|=|\mu|=|\nu|=d$, so
in particular Kronecker coefficients are symmetric in their three indices.

Often one writes partitions in terms of {\it Young diagrams}, where $\pi=(p_1\hd p_t)$
is represented by a collection of boxes, left justified,  with $p_j$ boxes in the $j$-th row.
There is a nice pictorial recipe for computing Littlewood-Richardson coefficients
in terms of Young diagrams (see, e.g., \cite[Chap. 5]{MR1464693}).

\begin{figure}[htb]
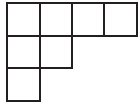

\[
\yng(4,2,1)
\]
\caption{Young diagram for $\pi=(4,2,1)$} \label{fig:Youngdiagram}
\end{figure}

A useful special case of the Littlewood Richardson coefficients  is the
  {\it Pieri formula }
\be\label{pieri}
 c^{\nu}_{\l, (d)}=\left\{\begin{matrix}
1 &{\rm if}\ \nu\ {\rm is\ obtained\ from\ }\l{\rm\ by\ adding\ }d{\rm \ boxes\  to }\\
&{\rm  the\ rows\ of\ } \l {\rm \ with\ no\ two\ in\ the\ same\ column;}\\
0&{\rm otherwise.}
\end{matrix}\right .
\ene
 
\exerone{\label{SLexer} Show that $(S^nV)^{\ot d}$ does not contain any $SL(V)$-invariants for $d<\bv$.}

\exerone{ Show that $k_{\pi\mu\nu}=\tdim (\thom_{\FS_d}([\pi],[\mu]\ot [\nu]))$.}

\subsection{A duality theorem for weight zero spaces and plethysms}
For any $S_{\pi}V$, the Weyl group  $\FS_{\bv}$ acts on the $\fsl$-{\it weight zero space}, which I will denote $(S_{\pi}V)_0$.
This is by definition the subspace of $S_{\pi}V$ on which the torus  $T_{V}$ acts trivially.  
Recall that $\FS_d$ acts on $V^{\ot d}$ and is the commutator of $GL(V)$. 

\exerone{Show that the weight zero space of $V^{\ot d}$ is 
  zero unless
${\bv}$ divides $d$, in which case we write $d={\bv}s$.}

Note that $S_{\mu}(S^sV)\subset V^{\ot s|\mu|}$. We have the following   duality theorem: 
\begin{theorem}\cite{MR0414794}\label{gaythm} For $|\pi|=d={\bv}s$ and $|\mu|={\bv}$, 
$$
\tmult_{\FS_{\bv}}([\mu], (S_{\pi}V)_0)= \tmult_{GL(V)}(S_{\pi}V, S_{\mu}(S^sV)).
$$
\end{theorem}

In particular, 
\begin{corollary}\label{gaycor} Let $|\pi|=d$.

\begin{enumerate}
\item When $d={\bv}$, $(S_{\pi}V)_0=[\pi]$. 

\item For any $d={\bv}s$, $\tdim [(S_{\pi}V)_0]^{\FS_d}=\tmult(S_{\pi}V,S^{\bv}(S^sV))$.
\end{enumerate}
\end{corollary}

To get an idea of the proof, note that
$S_{\pi}V=\thom_{\FS_d}([\pi], V^{\ot d})$ and thus  $(S_{\pi}V)_0=\thom_{\FS_d}([\pi], (V^{\ot d})_0)$, so the left hand
side is $\tmult_{\FS_{\bv}}([\mu], \thom_{\FS_d}([\pi], (V^{\ot d})_0))$. On the other hand,
$S^sV=(V^{\ot s})^{\FS_s}$, and $(S^sV)^{\ot {\bv}}=(V^{\ot {\bv}s})^{\FS_s\ctimes \FS_s}$, where we have ${\bv}$ copies of $\FS_s$.
So the right hand side is $\tmult_{\FS_d}([\pi], \thom_{\FS_{\bv}}([\mu], (V^{\ot {\bv}s})^{\FS_s\ctimes \FS_s})$. 
Now   $(V^{\ot d})_0$  is an $\FS_d$ and an $\FS_{\bv}$-module and   has a basis
$e_{i_1}\otc e_{i_d}$ with $\{ i_1\hd i_d\}=[{\bv}]^s$ where $d={\bv}s$.
Moreover  the $\FS_d$ and $\FS_{\bv}$ actions
commute, and the $\FS_d$ action is transitive on the set of basis elements. 
Letting $H=\FS_s^{\times {\bv}}\subset \FS_d$, D. Gay shows the normalizer
of $H$ divided by $H$ is $Nor(H)/H=\FS_{\bv}$ and the centralizer of $\FS_d$ in $\FS_{\tdim (V^{\ot d})_0}$ is $\FS_{\bv}$. 
The result follows by applying a combination of Frobenius reciprocity and
Schur-Weyl duality to go from $\FS_d$-modules to  $GL(V)$-modules. 
A key point is noting that $H$ is also the stabilizer of
  the vector
$
x:=e_1\otc e_1\ot e_2\otc e_2\otc e_{\bv}\otc e_{\bv}=e_1^{\ot s}\otc e_{\bv}^{\ot s} 
$.
 
\exerone{\label{snssv} We may realize  $S_{(s^{\bv})}V$ as $\BC\{ (e_1\ww\cdots \ww e_{\bv})^{\ot s}\}$. Show that $\FS_{\bv}$ acts
on $S_{(s^{\bv})}V$ by the sign representation when $s$ is odd and acts trivially when $s$ is even. 
Conclude $S^{\bv}(S^sV)$ has a unique $SL(V)$-invariant when $s$ is even and none when $s$ is odd,
and that $[S^{2\bv}(S^sV)]^{SL(V)}\neq 0$ for all $s>1$.
This had been observed in \cite[Prop. 4.3a]{MR983608}.}

\exerone{Show that the $SL(V)$-invariant $P\in S^{\bv}(S^sV)$ from the previous problem has the following expression.
Let $z=(x^1_1\cdots x^1_{s})\cdots (x^{\bv}_1\cdots x^{\bv}_{s})$.
Then
\be\label{polyexpr}
\langle \ol{P},  z\rangle =\sum_{\s_1\hd \s_{\bv} \in \FS_s} 
\ol{\tdet_{\bv}} (x^1_{\s_1(1)}\hd x^{\bv}_{\s_{\bv}(1)})\cdots 
\ol{\tdet_{\bv}} (x^s_{\s_1(s)}\hd x^{\bv}_{\s_{\bv}(s)}).
\ene
To compute $P(u)$ for any $u\in S^sV^*$, consider $u^{\bv}$ and expand it out as a sum of terms of the form $z$. Then
$P(u)=\langle \ol{P}, u^{\bv}\rangle$.
}

\section{Lower bounds via geometry}\label{lowgeomsect}

\subsection{The second fundamental form and the $\frac{m^2}2$ bound for Valiant's conjecture}\label{MRsect}
For hypersurfaces in  affine space, one can attach a differential invariant, the {\it second fundamental form},  to
each point. This form  is essentially the quadratic term in an adapted Taylor series for the hypersurface graphed over its
tangent space  at that point. The rank of this
quadratic form gives an invariant that can only decrease on the image of  general points under affine linear projections.
It is straight-forward to compute that for smooth points of $\{ \tdet_n=0\}$ the rank of the quadratic
form is $2n-2$ whereas, if one chooses a judicious point of $\{\tperm_m=0\}$ one finds the rank
is the maximal $m^2-2$. Combining these two gives:
\begin{theorem} \cite{MR2126826} $dc(\tperm_m)\geq \frac{m^2}2$, i.e.,  if $\tperm_m\in\tend(\BC^{n^2}\cdot \tdet_n)$, then $n\geq\frac{m^2}2$.
\end{theorem}
Valiant's conjecture \cite{vali:79-3} that motivated the work of Mulmuley and Sohoni is that $n$ must grow faster than
any polynomial in $m$ to have $\tperm_m\in\tend(\BC^{n^2}\cdot \tdet_n)$.

\subsection{Dual varieties and the $\frac{m^2}2$ lower bound for the Conjecture \ref{msmainconj}}\label{MR3048194sect}

Define $\Dual_{k,d,N}\subset \PP(S^dW^*)$ as the Zariski closure of the set of 
irreducible hypersurfaces of degree $d$ in $\PP W\simeq \PP^{N-1}$  whose dual variety has dimension at most $k$. 
(I identify a hypersurface (as a scheme) with its equation.)

It had been a classically studied problem to determine set-theoretic equations for $\Dual_{k,d,N}$. Motivated by GCT,
Manivel, Ressayre and I  were led to solve it. I follow \cite{MR3048194} in this subsection.

Let  $P\in S^dW^*$ be irreducible. The B. Segre  dimension formula \cite{MR0041481} states that for $[w]\in Z(P)_{general}$, 
$$\dim Z(P)\dual  = \trank (P_{d-2,1,1}(w^{d-2}))-2.
$$
The bilinear form  $P_{d-2,1,1}(w^{d-2})$ is called the {\it Hessian} of $P$ at $w$. 
Write $H_P$ for $P_{d-2,1,1}$; in bases it is  an $n\times n$ symmetric matrix whose entries are polynomials
of degree $d-2$. 

Thus  $\tdim (Z(P)\dual)\leq k$ if and only if,
for all $w\in \hat Z(P)$ and   $F\in G(k+3,W)$, 
$$
\tdet_{k+3} (H_P(w)|_F) =0.
 $$

Equivalently,    $P$ must divide   $\det_{k+3} (H_P|_F)$.
Note that $P\mapsto \det_{k+3} (H_P|_F)$ is a polynomial of degree $(k+3)(d-2)$ on $S^dW^*$.

By restricting first to a projective line $L\subset \BP W$, and then to an   affine line $\BA^1\subset   L$ within the projective
line, one can test divisibility by Euclidean division. The remainder will depend on our choice of coordinates on $\BA^1$, but
the  leading coefficient of the remainder only depends on the choice of point in $L$ that distinguishes the  affine line.

Set theoretically, the   equations obtained from the invariant
part of the remainder as one varies $\BA^1,L,F$  suffice to define $\Dual_{k,d,N}$   on the open 
subset parameterizing irreducible hypersurfaces, as  once the plane 
$\hat L$ is fixed, by varying the line $\BA^1$ one obtains a family of equations expressing
the condition that $P|_L$ divides $\det (H_{P}|_F)|_L$. A polynomial $P$ divides $Q$ if and only if    when restricted to each
plane $P$ divides $Q$, so the  conditions   imply that the dual variety of the irreducible hypersurface $Z(P)$ has dimension 
at most   $k$.

 By keeping track of weights along the flag
$\BA^1\subset \hat L^2\subset  F^{k+3}$ one concludes:

\begin{theorem}\label{degdualeqns}\cite{MR3048194}
The variety $\Dual_{k,d,N}\subset \PP(S^d(\BC^N)^*)$ has equations given by a copy of the $GL_N$-module 
 $S_{\pi(k,d)}\BC^N$,  where 
$$\pi(k,d) = ((k+2)(d^2-2d)+1,\, d(k+2)-2k-3 ,\, 2^{k+1}).   
$$
Since $|\pi|=d(k+2)(d-1)$, 
these equations have degree $(k+2)(d-1)$.
\end{theorem}


 If $P$ is not reduced, then these equations can vanish even if the
dual of the reduced polynomial with the same zero set as $P$ is non-degenerate. For example, if $P=R^2$ where
$R$ is a quadratic polynomial of rank $2s$, then $\tdet(H_P)$ is a multiple of $R^{2s}$.
The polynomial $\ell^{n-m}\tperm_m$ is  neither reduced nor irreducible, but fortunately we have
the following lemma:

\begin{lemma}\label{linfaclem}\cite{MR3048194}  Let $U=\BC^M$ and $L=\BC$, let  $R\in S^m U^*$ be  irreducible,
let $\ell\in L^*$ be nonzero,  let $U^*\op L^*\subset W^*$ be a linear inclusion,
and let $P=\ell^{d-m}R\in S^dW^*$. 

If $[R]\in \Dual_{\kappa, m, M}$ and 
$[R]\not\in \Dual_{\kappa-1, m, M}$, then $[P]\in \Dual_{\kappa, d, N}$ and 
$[P]\not\in \Dual_{\kappa-1, d, N}$.
\end{lemma}

Checking that $\{\tperm_m=0\}\dual$  is indeed a hypersurface by computing the second fundamental form
of $\{\tperm_m=0\}$ is of full rank 
at the matrix all of whose entries are $1$ except, e.g.,  the $(1,1)$ slot which is $1-n$   (the kernel of the second fundamental
form has the same dimension as the kernel of the Hessian),  
it follows $\Perm^m_{\frac {m^2} 2}\not\subset \Det_{\frac {m^2}2}$
proving Theorem \ref{LMRbnd}.

The main theorem of \cite{MR3048194} is:

\begin{theorem}\label{smooth}\cite{MR3048194} 
The scheme $\Dual_{2n-2,n,n^2}$ is smooth at $[\det_n]$, and $\Det_n$  is an irreducible component of $\Dual_{2n-2,n,n^2}$. 
\end{theorem}

 For   polynomials in $N'<N$ variables, the maximum rank of the Hessian
is $N'$ so   the determinant of the Hessian will vanish on
any $F$ of dimension $N'+1$. Thus  $Sub_{k+2}(S^nW)\subset \Dual_{k,n,N}$. The     subspace variety  $Sub_{k+2}(S^d\BC^N)$, which  
 has dimension $\binom{k+d+1}d+(k+2)(N-k-2)-1$,  also forms an irreducible component of $\Dual_{k,n,N}$ (see \cite{MR3048194}), 
 so  $\Dual_{2n-2,n,n^2}$ is not   irreducible.

Theorem \ref{smooth} is proved by computing the Zariski tangent space to both varieties at $[\tdet_n]$. To carry
out the computation, one uses that the Zariski tangent spaces are $G_{\tdet_n}$-modules, so
one just needs to single out a vector in each $S_{\pi}E\ot S_{\pi}F$.  On then uses  immanants (see \S\ref{permbeauty}) to get a preferred vector in each module to test.

In particular, Theorem \ref{smooth} implies that the $GL(W)$-module of highest weight $ \pi(2n-2,n) $ given by 
Theorem \ref{degdualeqns}
gives  local equations at $[\det_n]$
of $\ol{GL_{n^2}\cdot [\tdet_n]}$, of degree $2n(n-1)$.

\section{The boundary of $\Det_n$}  \label{bndrysect}
It is expected that understanding 
the components of the boundary of $\Det_n$
will be useful for GCT. There is the
obvious component obtained by eliminating a variable, which is contained in $\tend(\BC^{n^2}\cdot \tdet_n)$, and is
related to Valiant's conjecture. To understand the difference between Valiant's conjecture and the Conjecture \ref{msmainconj},
one needs to examine the other components of the boundary.

Determining additional components of the boundary relates to yet another classical question: determine unextendable linear
spaces on the hypersurface $\{\tdet_n=0\}$. Roughly speaking, given one such, call it $L\subset \BC^{n^2}$, write
$\BC^{n^2}=L\op L^c$ where $L^c$ is some choice of complement to $L$. Then compose the determinant with a 
(suitably normalized) curve 
$f_t=Id_L+ tId_{L^c}\in GL_{n^2}$. In the limit as $t\ra 0$ one may arrive at a new component of the boundary.

For an explicit example, write $\BC^{n^2}=W=W_S\op W_{\Lambda}$, where we split up the $n\times n$ matrices into symmetric and skew-symmetric matrices.
When $n$ is odd, the curve
$$
g(t)=\frac 1t(Id_{W_{\Lambda}}+tId_{W_S})
$$
determines  a   polynomial $P_{\Lambda}:=\tlim_{t\ra 0}g(t)\cdot \tdet_n$.
To see 
$P_{\Lambda}$ explicitly, 
decompose a matrix $M$ into its symmetric and skew-symmetric parts 
$M_S$ and $M_{\Lambda}$. Then
$$
P_{\Lambda}(M)= \ol{\tdet_n}(M_{\Lambda},\ldots, M_{\Lambda},M_S). 
$$
More explicitly, $P_{\Lambda}$ can   be expressed as follows. 
Let $Pf_i(M_{\Lambda})$ denote the Pfaffian of the skew-symmetric matrix, of even size, 
obtained from $M_{\Lambda}$ by suppressing its $i$-th row and column. Then 
$$P_{\Lambda}(M)=\sum_{i,j}(M_S)_{ij}Pf_i(M_{\Lambda})Pf_j(M_{\Lambda}).$$

\begin{prop}\label{Lambda}\cite{MR3048194}
The orbit closure 
$\ol{GL(W)\cdot P_{\Lambda}}$ is an irreducible codimension one component of 
$\partial \Det_n$ that is  not contained in $\tEnd(W)\cdot [\tdet_n]$. In particular $\ol{dc}(P_{\Lambda,m})=m<dc(P_{\Lambda,m})$.
\end{prop}

Proposition \ref{Lambda} indicates  that   Conjecture \ref{msmainconj} could be  strictly stronger
than Valiant's conjecture.
To prove the   second assertion, one
computes the stabilizer $G_{P_{\Lambda}}$ explicitly and sees  it has dimension one less than 
the dimension of $G_{\tdet_n}$.

 The hypersurface   $Z(P_\Lambda)\subset\BP W$ has interesting properties, for example:  

\begin{prop}\label{Lambdaex}\cite{MR3048194}
$$Z(P_\Lambda)\dual =\ol{\PP \{v^2\oplus v\wedge w\in S^2\BC^n\op \La 2\BC^n, \; v,w\in\CC^n\}}
\subset \PP W^*.$$
\end{prop}

Note that  $Z(P_\Lambda)^\dual$ resembles      $Seg(\PP^{n-1}\times \PP^{n-1})$. 
It can be defined as the image of the projective bundle 
$\pi : \PP(E)\rightarrow\PP^{n-1},$
where $E={\mathcal O}(-1)\oplus Q$ is the sum of the
tautological and quotient bundles on $\PP^{n-1}$, by a sub-linear system of 
${\mathcal O}_E(1)\otimes \pi^*{\mathcal O}(1)$. This sub-linear system
contracts the divisor $\PP(Q)\subset \PP(E)$ to the Grassmannian 
$G(2,n)\subset \PP \La 2\BC^n$.

\medskip

The only other components of $\partial \Det_n$ that I am aware of were found by 
  J. Brown, N. Bushek, L. Oeding, D. Torrance and Y. Qi, as part of an AMS Mathematics Research Community in June
2012. They found two additional components of $\partial \Det_4$.

\begin{problem} Find additional components of $\partial \Det_n$.\end{problem}

\begin{problem} Determine all  components of $\partial \Det_3$.\end{problem}

\section{Symmetries of polynomials and coordinate rings of orbits}\label{symsect}
Throughout this section $G=GL(V)$ and $\tdim V=n$. 
Given $P\in S^dV$, let 
$$G_P:=\{ g\in GL(V)\mid g\cdot P=P\}=\{ g\in GL(V)\mid P(g\cdot x)=P(x)\forall x\in V^*\}
$$ 
denote
the symmetry group of $P$. We let $G_{[P]}:=\{ g\in GL(V)\mid g\cdot [P]=[P]\}$.
Determining the connected component of the identity $G_{P}^0$ is simply a matter of linear algebra,
as the computation of $\fg_P$ is a linear problem. However one can compute $G_P$ directly in only a few simple cases.

Throughout this section, let $V=\BC^n$ and use index ranges $1\leq i,j,k\leq n$.  
Examples \ref{eex1}, \ref{eex2}, \ref{fermatex}, and \ref{sumprodex} follow \cite{MR2901512}.

\subsection{Two easy examples}
\begin{example} \label{eex1} Let $P=x_1^d\in S^dV$. Let $g=(g^i_j)\in GL(V)$. Then $g\cdot (x_1^d)= (g^j_1x_j)^d$ so
if $g\cdot (x_1^d)=x_1^d$, then  $g^j_1=0$ for $j>1$ and $g^1_1$ must be a $d$-th root of unity. There are no other restrictions, thus
$$
G_P=\left\{ g\in GL(V)\mid  g=\begin{pmatrix} \o & * & \cdots & *\\ 0& *&\cdots &*\\ & \vdots & &\\ 0& *&\cdots &*\end{pmatrix}, \o^d=1
\right\}, \ 
G_{[P]}=\left\{ g\in GL(V)\mid  g=\begin{pmatrix} * & * & \cdots & *\\ 0& *&\cdots &*\\ & \vdots & &\\ 0& *&\cdots &*\end{pmatrix}
\right\}.
$$
\end{example}

The $GL(V)$ orbit of $[x_1^d]$ is closed and equal to the Veronese variety $v_d(\BP V)$.
 
\exerone{Use Corollary \ref{orbitcor} to determine $\BC[v_d(\BP V)]$ (even if you already know it by a different method).}

\begin{example}\label{eex2} Let $P=\tchow_n=x_1\cdots x_n\in S^nV$, which I will call the \lq\lq Chow polynomial\rq\rq .
It is clear $\G_n:=T_n^{SL}\rtimes \FS_n\subset G_{\tchow_n}$, we need to determine if the stabilizer is larger.
Again, we can work by brute force:
$g\cdot \tchow_n= (g^j_1x_j)\cdots (g^j_nx_j)$. In order that this be equal to $x_1\cdots x_n$, by unique
factorization of polynomials, there must be a permutation $\s\in \FS_n$
such that for each $k$, we   have $g^j_kx_j=\l_k x_{\s(k)}$ for some $\l_k\in \BC^*$. Composing with the
inverse of this permutation we have $g^j_k=\d^j_k\l_j$, and finally we see that
we must further have $\l_1\cdots \l_n=1$, which means it is an element of $T_n^{SL}$, so the
original $g$ is an element of $\G_n$.

The orbit closure of $\tchow_n$ is the Chow variety $Ch_n(V)\subset \BP S^nV$. The coordinate
ring of $GL(V)\cdot \tchow_n$ is discussed in \S\ref{chowsect}.
\end{example}

\subsection{Techniques}
We can usually guess a large part of $G_P$. We then
 form auxiliary objects from $P$ which   have a symmetry group $H$ that one can compute, and by construction
$H$  contains $G_P$. If $H=G_P$, we are done, and if not, we simply have to examine the
difference between the groups.

\begin{remark} The very recent preprint \cite{skipsym} describes further techniques for determining stabilizers of points.
\end{remark}

Consider the hypersurface $Z(P):=\{ [v]\in \BP V^*\mid P(v)=0\} \subset \BP V^*$. If all the irreducible components of
$P$ are reduced, then $G_{Z(P)}=G_{[P]}$, as a reduced polynomial may be recovered up to scale from
its zero set, and in general $G_{Z(P)}\supseteq G_{[P]}$.
  Moreover,  we can consider
its singular set $Z(P)_{sing}$, which may be described as the zero set of the image of $P_{1,d-1}$ (which is essentially
the exterior derivative $dP$). If $P=a_Ix^I$, where $a_{i_1\hd i_d}$ is symmetric in its lower indices,  then $Z(P)_{sing}=\{ [v]\in \BP V^*\mid
a_{i_1,i_2\hd i_d}x^{i_2}(v)\cdots x^{i_d}(v)=0\ \forall i_1\}$.
While we could consider the singular locus of the singular locus etc.., it turns out to be easier
to work with what I will call the {\it very singular} loci. 
For an arbitrary variety $X\subset \BP V$, define $X_{verysing}=X_{verysing,1}:=\{ x\in \BP V\mid dP_x=0 \forall P\in I(X)\}$.
If $X$ is a hypersurface, then $X_{sing}=X_{verysing}$ but in general they can be different.
Define $X_{verysing,k}:=(X_{verysing,k-1})_{verysing}$. Algebraically,
if $X=Z(P)$ for some $P\in S^dV$, then the ideal of  $Z(P)_{verysing,k}$ is generated by the image
of $P_{k,n-k}: S^kV^*\ra S^{n-k}V$. 
   The symmetry groups of these varieties all contain $G_P$.

\subsection{The Fermat}\label{fermatex}
Let $\tfermat_n^d:=x_1^d+\cdots +x_n^d$. The $GL(V)$-orbit closure of   $[\tfermat^d_n]$ is the $n$-th secant variety of the
Veronese variety $\s_n(v_d(\BP V))\subset \BP S^nV$, see \S\ref{relevantsect}.
It is clear $\FS_n \subset G_{\tfermat}$, as well as the diagonal matrices whose entries are $d$-th roots of unity. We need
to see if there is anything else. The first idea, to look at the singular locus, does not work, as the zero set is
smooth, so we consider
$\tfermat_{2,d-2}=x_1^2\ot x^{d-2}+\cdots +x_n^2\ot x^{d-2}$. Write the further
polarization  $P_{1,1,d-2}$ as a symmetric matrix whose
entries are homogeneous polynomials of degree $d-2$ (the Hessian matrix).
We get
$$
\begin{pmatrix} x_1^{d-2} & & \\ & \ddots & \\ & & x_n^{d-2}\end{pmatrix}.
$$
Were   the determinant of this matrix $GL(V)$-invariant, we could proceed as we did with $\tchow_n$, using unique factorization.
Although it is not, it is close enough as follows:
Recall that for a linear map $f: W\ra V$, where $\tdim W=\tdim V=n$,  we have $f^{\ww n}\in \La nW^*\ot \La n V$
and an element   $(h,g)\in GL(W)\times GL(V)$ acts on $f^{\ww n}$ by $(h,g)\cdot f^{\ww n}=(\tdet(h))\inv (\tdet(g)) f^{\ww n}$.
In our case $W=V^*$ so   $  P_{2,d-2}^{\ww n}(x)=\tdet(g)^2  P_{2,d-2}^{\ww n}(g\cdot x)$, and 
the polynomial obtained by the determinant of the Hessian matrix is invariant up to scale.

Arguing as above,   $(g^j_1x_j)^{d-2}\cdots (g^j_nx_j)^{d-2}= x_1^{d-2}\cdots x_n^{d-2}$ and we conclude again by
unique factorization that $g$ is in $\G_n$. Composing with a permutation matrix to make $g\in T$, we   see
that, by acting on the Fermat itself, that the entries on the diagonal are $d$-th roots of unity.

\exerone{Show that the Fermat is characterized by its symmetries.}

\subsection{The sum-product polynomial}\label{sumprodex}
The following polynomial, called the {\it sum-product polynomial},  will be important when studying depth-$3$ circuits. 
Its $GL(mn)$-orbit
closure is the $m$-th secant variety of the Chow variety $\s_m(Ch_n(\BC^{nm}))$:
$$
S^n_m:=\sum_{i=1}^m\Pi_{j=1}^n x_{ij}\in S^n(\BC^{nm}).
$$

\exerone{Determine $G_{S^n_m}$ and show that $S^n_m$ is characterized by its symmetries.}

\subsection{The determinant}
  I follow   \cite{MR0029360} in this section. Write $\BC^{n^2}=E\ot F$ with $E,F=\BC^n$. 

\begin{theorem}[Frobenius \cite{Frobdet}]\label{frobdet} Let $\phi\in \rho(Gl_{n^2})\subset GL(S^n\BC^{n^2})$ be such that $\phi(\tdet_n)=\tdet_n$.
Then, identifying  $\BC^{n^2}\simeq Mat_{n\times n}$, 
$$
\phi(X)=\left\{\begin{matrix} X\mapsto gXh     \\
X\mapsto gX^Th\end{matrix}   \right.
$$
where $g,h\in GL_n$,   and $\tdet_n(g)\tdet_n(h)=1$. Here $X^T$ denotes the transpose of $X$.
\end{theorem}

\begin{corollary} Let $\mu_n$ denote the $n$-th roots of unity embedded diagonally in $SL(E)\times SL(F)$. Then 
$G_{\tdet_n}=(SL(E)\times SL(F))/\mu_n\rtimes \BZ_2$
\end{corollary}

To prove the Corollary, just note that the $\BC^*$ and $\mu_n$ are in the kernel of the map $\BC^*\times SL(E)\times SL(F)\ra GL(E\ot F)$.

\exerone{\label{n=2case} Prove the $n=2$ case of the theorem. Hint: in this case the determinant is a smooth
quadric.}

Write $\BC^{n^2}=W=A\ot B=\thom(A^*,B)$.
The following lemma is standard, its proof is left as an exercise:

\begin{lemma}\label{stdlem} Let $U\subset  W$  be a linear subspace such that 
$U \subset \{\tdet_n=0\}$.  Then $\tdim U\leq n^2-n$ and the subvariety of the Grassmannian 
$G(n^2-n,W)$ consisting  of maximal linear spaces on $\{\tdet_n=0\}$ has two components, call them
$\Sigma_{\a}$ and $\Sigma_{\b}$, where
\begin{align}
\Sigma_{\a}&=\{ X\mid \tker (X)= \hat L {\rm \ for\ some\ }L\in \BP A\}, and \\
\Sigma_{\b}&=\{ X\mid \tim (X)= \hat H {\rm \ for\ some\ }H\in \BP B^*\}.
\end{align}
Moreover, for any two distinct $X_{j}\in \Sigma_{\a}$, $j=1,2$,  and $Y_j\in \Sigma_{\b}$ we have
\begin{align}
\label{leme1}\tdim (X_1\cap X_2)=\tdim (Y_1\cap Y_2)&=n^2-2n, and \\
\label{leme2} \tdim(X_i\cap Y_j)=n^2-2n+1. \end{align}
\end{lemma}

\begin{proof}[Proof of theorem \ref{frobdet}]
Let $\Sigma=\Sigma_{\a}\cup \Sigma_{\b}$. Then   
the map on $G(n^2-n,W)$ induced by $\phi$ must preserve $\Sigma$.
By the conditions \eqref{leme1},\eqref{leme2} of Lemma \ref{stdlem}, 
in order
to preserve dimensions of intersections,
every
$X\in \Sigma_{\a}$ must map to a point of $\Sigma_{\a}$ or every 
$X\in \Sigma_{\a}$ must map to a point of  $\Sigma_{\b}$, and similarly for $\Sigma_{\b}$. 
If we are in the second case, replace $\phi$ by $\phi \circ T$, where
$T(X)=X^T$, so we may now assume $\phi$ preserves both $\Sigma_{\a}$ and $\Sigma_{\b}$.

Now $\Sigma_{\a}\simeq \BP A$, so $\phi$ induces an algebraic map $\phi_A: \BP A\ra \BP A$.
If $L_1,L_2,L_3\in \BP A$ lie on a $\pp 1$, in order for $\phi$ to preserve the dimensions of triple
intersections, the images of the $L_j$ under $\phi_A$ must also lie on a $\pp 1$.
By Exercise \ref{n=2case} we may assume $n\geq 3$ so the above condition is non-vacuous.
But then, by classical   projective geometry $\phi_A\in PGL(A)$, and
similarly, $\phi_{B}\in PGL(B)$, where $\phi_{B}: \BP B^*\ra \BP B^*$ is the corresponding
map. Write $\hat\phi_{A}\in GL(A)$ for any choice of lift and similarly for $B$. 

Consider the map $\tilde \phi \in \rho(GL(W))$ given by $\tilde\phi(X)={\hat \phi_A}\inv \phi (X){\hat \phi_B}\inv$.
The map $\tilde\phi$ sends each $X_j\in \Sigma_{\a}$ to itself as well as each $Y_j\in \Sigma_{\b}$, in particular it does the same
for all intersections. Hence it preserves $Seg(\BP A\times \BP B)\subset \BP (A\ot B)$ point-wise, so
it is up to scale the identity map.
\end{proof}

\begin{remark} For those familiar with Picard groups, M. Brion points out that there is a shorter proof 
of Theorem \ref{frobdet} as follows: In general, if a polynomial $P$ is
reduced and irreducible, then $G_{Z(P)\dual}=G_{Z(P)}=G_{[P]}$. (This follows as $(Z(P)\dual)\dual=Z(P)$.)
The dual of $Z(\det_n)$ is the Segre $Seg(\pp{n-1}\times \pp{n-1})$. Now the automorphism group of
$\pp{n-1}\times \pp{n-1}=\BP E\times \BP F$ acts on the Picard group which is $\BZ\times \BZ$ and preserves the two
generators $\cO_{\BP E\times \BP F}(1,0)$ and $\cO_{\BP E\times \BP F}(0,1)$ coming from the generators
on $\BP E,\BP F$. Thus, possibly composing with $\BZ_2$ swapping the generators (corresponding to transpose
in the ambient space), we may assume each generator is preserved. But then we must have
an element of $Aut(\BP E)\times Aut(\BP F)=PGL(E)\times PGL(F)$. Passing back to the ambient space, we obtain the result.
\end{remark}

\subsection{The coordinate ring of $GL(W)\cdot \tdet_n$}\label{detoring}
For   $\pi=(p_1\hd p_{n^2})$ with $p_1\geq \cdots \geq p_{n^2}$, recall that the multiplicity of $S_{\pi}W$ in
$\BC[GL(W)\cdot \tdet_n]$ is $\tdim (S_{\pi}W)^{G_{\tdet_n}}$,
where   $G_{\tdet_n}=(SL(E)\times SL(F))/\mu_n\rtimes \BZ_2$. Following \S\ref{glwreps}, write
the $SL(E)\times SL(F)$-decomposition 
$S_{\pi}(E\ot F)=\oplus (S_{\mu}E\ot  S_{\nu}F)^{\op k_{\pi\mu\nu}}$.
To have $SL(E)\times SL(F)$-trivial modules, we need $\mu=\nu=(\d^n)$ for some $\d\in \BZ$.
Recalling the interpretation $k_{\pi\mu\nu}=\tdim (\thom_{\FS_d}([\pi],[\mu]\ot [\nu])$, 
when $\mu=\nu$, $[\mu]\ot [\mu]=S^2[\mu]\ot \La 2[\mu]$. 
Define the {\it symmetric Kronecker coefficient} $sk^{\pi}_{\mu\mu}: = 
\dim \Hom_{\FS_d}([\pi],S^2 [\mu])$. 
It is not hard to check (see \cite{MR2861717})
that    
$$((S_{\d^n}E\ot S_{\d^n}F)^{\op k_{\pi\d^n\d^n}})^{\BZ_2}= (S_{\d^n}E\ot S_{\d^n}F)^{\op sk^{\pi}_{\d^n\d^n}}.
$$

We conclude: 

\begin{proposition}\label{peterrefx}\cite{MR2861717} Let $W=\BC^{n^2}$. 
The coordinate ring of the $GL(W)$-orbit  of $\tdet_n$ is
$$
\BC[GL(W)\cdot  \tdet_n]=
 \bigoplus_{\d\in \BZ}\bigoplus_{\pi \, \mid \, |\pi|=n\d }(S_{\pi}W^*)^{\op sk^\pi_{\d^n\d^n}}.
 $$
  \end{proposition}

Thus   partitions $\pi$ of $dn$  such that $sk^\pi_{\d^n\d^n}< \tmult(S_{\pi}W, S^d(S^nW))$ 
are representation-theoretic obstructions, and if moreover $\tmult(S_{\pi}W, S^d(S^nW))=0$, $S_{\pi}W$  is an occurrence
obstruction. C. Ikenmeyer \cite{ikenthesis} has examined the situation for $\Det_3$. He found on the order of
$3,000$ representation-theoretic obstructions, of which on the order of $100$ are occurrence obstructions in
degrees up to $d=15$. There are two such partitions with seven parts,
$(13^2,2^5)$ and $(15,5^6)$. The rest  consist of partitions with at least $8$ parts (and many with
$9$). Also of interest is that for approximately $2/3$ of the partitions
  $sk^\pi_{\d^3\d^3}<k_{\pi \d^3\d^3}$.
The lowest degree of  an occurrence obstruction is $d=10$, where $\pi=(9^2,2^6)$ has 
$sk^\pi_{10^3 10^3}=k_{\pi 10^3 10^3}=0$ but $\tmult(S_{\pi}W, S^{10}(S^3W))=1$.
In degree $11$,  $\pi=(11^2,2^5,1)$ is an occurrence obstruction where $\tmult(S_{\pi}W, S^{11}(S^3W))=k_{\pi 11^3 11^3}=1>0=sk^\pi_{11^3 11^3}$.

\subsection{The permanent}\label{permsect}
Write $\BC^{n^2}=E\ot F$. Then  it is easy to see
$(\G_n^E\times \G_n^F)\rtimes \BZ_2\subseteq G_{\tperm_n}$, where the nontrivial element of  $\BZ_2$ acts by sending a matrix
to its transpose and recall $\G^E_n=T_E\rtimes \FS_n$ . We would like to show this is the entire symmetry group. However,  it is not when $n=2$.

\exerone{What is $G_{\tperm_2}$? Hint: $\{\tperm_2=0\}$ is a smooth quadric.}

\begin{theorem}\cite{MR0137729}\label{permstabthm}  For $n\geq 3$, $G_{\tperm_n}=(\G_n^E\times \G_n^F)/\mu_n\rtimes \BZ_2$.
\end{theorem}

\begin{remark}\label{permbeauty} From Theorem \ref{permstabthm}, one can begin to appreciate the beauty of the permanent. Since
$\tdet_n$ is the only polynomial invariant under $SL(E)\times SL(F)$, to find other interesting
polynomials on spaces of matrices, we will have to be content with subgroups of this group.
But what could be a more natural subgroup than the product of the normalizer of the tori?
In fact, say we begin by asking  simply for a polynomial invariant under the action of $T_E\times T_F$.
We need to look at $S^n(E\ot F)_0$, where the $0$ denotes the $\fsl$-weight zero subspace.
This decomposes as $\oplus_{\pi} (S_{\pi}E)_0\ot (S_{\pi}F )_0$. By Corollary \ref{gaycor}(i), these spaces 
are the $\FS_n^E\times\FS_n^F$-modules $[\pi]\ot [\pi]$. Only one of these is trivial, and that corresponds
to the permanent! More generally, if we consider the diagonal $\FS_n\subset \FS_n^E\times \FS_n^F$,
then both $[\pi]$'s are modules for the same group, and since $[\pi]\simeq [\pi]^*$, there is
then a preferred vector corresponding to the identity map. These vectors are Littlewood's
{\it immanants}, of which the determinant and permanent are special cases.
\end{remark} 

Consider $Z(\tperm_n)_{sing}\subset \BP (E\ot F)^*$. It
 consists of the matrices all of whose size $n-1$ submatrices have zero permanent. (To see
this,  note the permanent   has   Laplace type expansions.) This seems
even more complicated than the hypersurface $Z(\tperm_n)$ itself. Continuing, 
$Z(\tperm_n)_{verysing,k}$ consists of the matrices all of whose sub-matrices of size $n-k$ have zero permanent.
In particular $Z(\tperm_n)_{verysing,n-2}$ is defined by quadratic equations. Its zero set has many components,
but each component is easy to describe:

\begin{lemma} Let $A$ be an $n\times n$ matrix all of whose size $2$ submatrices have zero permanent. Then   one of
the following hold: 
\begin{enumerate}
\item   all the entries of $A$   are zero except those in a single size $2$ submatrix, and that submatrix has
zero permanent.
\item    all the entries of $A$   are zero except those in the $j$-th row for some $j$. Call the associated  component 
  $C^j$.
\item   all the entries of $A$   are zero except those in the $j$-th column for some $j$.
Call the  associated component
 $C_j$.
\end{enumerate}
\end{lemma}
The proof is straight-forward. Take a matrix with   entries that don't fit that pattern, e.g., one that begins
$$
\begin{matrix} a&b&e\\ *&d&*\end{matrix}
$$
and note that it is not possible to fill in the two unknown entries and have all size two sub-permanents, even in this corner,  zero.
There are just a few such cases since we are free to act by $\FS_n\times \FS_n$.

\begin{proof}[Proof of theorem \ref{permstabthm}] (I follow \cite{MR2807220}.) Any linear transformation preserving the permanent must send
a component of $Z(\tperm_n)_{verysing,n-2}$ of type (1) to another of type (1). It must send a component $C^j$ either to
some $C^k$ or some $C_i$. But if $i\neq j$,  $C^j\cap C^i=0$ and for all $i,j$,  $\tdim(C^i\cap C_j)=1$.
Since intersections must be mapped to intersections,  either all components  
$C^i$ are sent to components $C_k$ or all are permuted among themselves. By composing with an  element of  $\BZ_2$, 
we may assume all the $C^i$'s are sent to $C^i$'s and the $C_j$'s are sent to $C_j$'s. Similarly, by composing with an
element of $\FS_n\times \FS_n$ we may assume each $C_i$ and $C^j$ is sent to itself. But then their intersections are sent to
themselves.
So we have, for all $i,j$, 
\be\label{xijm}
(x^i_j)\mapsto (\l^i_jx^i_j)
\ene
for some $\l^i_j$ and there is no summation in the expression.
Consider the image of a size $2$ submatrix, e.g., 
\be\label{xijn}
\begin{matrix} x^1_1 & x^1_2\\ x^2_1 &x^2_2\end{matrix} \mapsto  \begin{matrix} \l^1_1 x^1_1 & \l^1_2x^1_2\\ \l^2_1x^2_1 &\l^2_2x^2_2\end{matrix}.
\ene
In order that the map \eqref{xijm} be in $G_{\tperm_n}$, when $(x^i_j)\in Z(\tperm_n)_{verysing,n-2}$, the
  permanent of the matrix on the right hand side of \eqref{xijn} must be zero, which implies $\l^1_1\l^2_2-\l^1_2\l^2_1=0$, thus all the
$2\times 2$ minors of the matrix $(\l^i_j)$ are zero, so it has rank one and is the product of a column vector and a row vector,
but then it is   an element of $T_E\times T_F$.
\end{proof}

\subsection{Iterated matrix multiplication}
Let $IMM^k_n\in S^n(\BC^{k^2n})$ denote the iterated matrix multiplication operator for $k\times k$ matrices, 
$(X_1\hd X_n)\mapsto \ttrace(X_1\cdots X_n)$. Letting $V_j=\BC^k$, invariantly 
\begin{align*} 
IMM^k_n=Id_{V_1}\otc Id_{V_n}\in &(V_1\ot V_2^*)\ot (V_2\ot V_3^*)\otc (V_{n-1}\ot V_n^*)\ot (V_n\ot V_1^*)\\
&\subset S^n((V_1\ot V_2^*)\op (V_2\ot V_3^*)\op \cdots \op  (V_{n-1}\ot V_n^*)\op (V_n\ot V_1^*)),
\end{align*}
and the connected component of the identity of $G_{IMM^k_n}\subset GL(\BC^{k^2n})$  is clear.

\begin{problem} Determine $G_{IMM^3_n}$.
\end{problem}

The case of $IMM^3_n$ is important as this sequence is complete for the complexity class $\vp_e$, see \S\ref{complexapp}.
Moreover $IMM^n_n$ is complete for the class $\vp_{ws}$. 

\begin{problem} Find equations in the ideal of  $\ol{GL_{9n}\cdot IMM^3_n}$.
Determine lower bounds for the inclusions $\Perm^m_n\subset \ol{GL_{9n}\cdot IMM^3_n}$ and
$\Det^m_n\subset \ol{GL_{9n}\cdot IMM^3_n}$.
\end{problem}

\section{The Chow variety}\label{chowsect}

If one specializes the determinant or permanent to diagonal matrices and takes the orbit closure, one
obtains the Chow variety defined in \S\ref{relevantsect}. Thus $I(\Det_n)\subset I(Ch_n(W))$. The ideal
of the Chow variety has been studied for some time, dating back at least to  Gordan and Hadamard. The history
is rife with rediscoveries and errors that only make the subject more intriguing.

The secant varieties of the Chow variety are also important
 for the study of depth $3$ circuits,
as described in \S\ref{depth3sect}.
 It is easy to see that
$\s_2(Ch_n(W))\subset \Det_n$, and a consequence of the equations described in \S\ref{MR3048194sect} is that 
$\s_3(Ch_n(W))$ is {\it not } contained in $\Det_n$. I do not know if it is contained in $\Perm^n_n$.

\begin{problem} Determine if $\s_3(Ch_n(W))\subset \Perm^n_n$.
\end{problem}

\begin{problem} Determine equations in the ideal of $\s_2(Ch_n(W))$. Which modules are in the ideal of $\Det_n$?
\end{problem}

\subsection{History}
A map, which, following a suggestion of A. Abdessalem, I now call the {\it Hermite-Hadamard-Howe} map, $h_{d,n}:S^d(S^nW)\ra S^n(S^dW)$
was defined by Hermite \cite{hermite} when $\tdim W=2$, and Hermite proved the map is an isomorphism
in this case. His celebrated reciprocity theorem (Theorem \ref{hermthm}) is this isomorphism. Hadamard \cite{MR1554881} defined the map in
general and observed that its kernel is $I_d(Ch_n(W^*))$, the degree $d$ component of the ideal
of the Chow variety (see \S\ref{chowideal}). Originally he mistakenly thought the map was always of maximal rank,
but in \cite{MR1504330} he proved the map is an isomorphism when $d=n=3$ and posed determining if
injectivity holds  in general when $d\leq n$ as a open problem. (Injectivity for $d\leq n$ is equivalent
to surjectivity when $d\geq n$, see Exercise \ref{hdualex}.)   Brill   wrote down set-theoretic equations
for the Chow variety of degree $n+1$, via a map that I denote
$Brill: S_{n,n}W\ot S^{n^2-n}W\ra S^{n+1}(S^nW)$, see \cite{gkz} or \cite{MR2865915}. There was
a   gap in Brill's argument, that was repeated in \cite{gkz} and finally fixed by E. Briand in
\cite{MR2664658}. The map $h_{d,n}$ was rediscovered by Howe in \cite{MR983608} where he
also wrote \lq\lq it is reasonable to expect\rq\rq\  that $h_{d,n}$ is always of maximal rank.
This reasonable expectation dating back to Hadamard has become known as the 
\lq\lq Foulkes-Howe conjecture\rq\rq .
Howe had been investigating a conjecture of Foulkes \cite{MR0037276} that for $d>n$, the
irreducible modules counted with multiplicity occurring in $S^n(S^dW)$ also occur in $S^d(S^nW)$.
Howe's conjecture is now known to be false, and Foulkes' original conjecture is still open.
An asymptotic version of Foulke's conjecture was proved by Manivel \cite{MR1651092}, and asymptotic
versions of Howe's conjecture by Brion  \cite{MR1243152,MR1601139} as discussed
below.
The proof that Howe's conjecture is false follows from a computer calculation of
M\"uller and Neunh\"offer \cite{MR2172706} related to the symmetric group.  
A. Abdessalem   realized their computation showed the map $h_{5,5}$ is not injective. (In 
\cite{MR2172706} they mistakenly  say the result comes from  \cite{Briand:these} rather than their own paper.)
This computation was mysterious, in particular, the modules in the kernel were
not determined. As part of an AMS Mathematics Research Community in June
2012 and follow-up to it,   C. Ikenmeyer and  S. Mkrtchyan determined the modules in the kernel explicitly.
In particular the kernel does not consist of isotypic components.
In his PhD thesis \cite{Briand:these}, Briand announced a proof that if $h_{d,n}$ is surjective, then
$h_{d+1,n}$ is also surjective. Then  A. Abdesselam found a gap in Briand's argument. Fortunately
this result follows  from results of T. McKay \cite{MR2394689}, see \S\ref{combsect}.
  Brion \cite{MR1243152,MR1601139}, and independently Weyman and Zelevinsky (unpublished)
 proved that 
the Foulkes-Howe conjecture is true asymptotically (see Corollary \ref{brionfhcor}), with Brion giving an explicit, but very large
bound for $d$ in terms of $n$ and $\tdim W$, see Equation \eqref{brionexpl}.

\begin{problem} What is the kernel of  $Brill: S_{n,n}W\ot S^{n^2-n}W\ra S^{n+1}(S^nW)$?
\end{problem}

\subsection{The ideal of the Chow variety}\label{chowideal}\label{choweqnsect} 
\index{Chow variety!equations of}
Consider the   map  $h_{d,n}: S^d(S^nW)\ra S^n(S^dW)$ defined as follows:
First  include $S^{d}(S^nW)\subset W^{\ot nd}$. 
Next,  regroup the copies of $W$ and symmetrize the blocks  to
$(S^{d}W)^{\ot n}$. Finally, thinking of $S^{d}W$ as a single
vector space, symmetrize again.

For example, putting  subscripts on $W$ to indicate   position:
\begin{align*}
S^2(S^3W)\subset W^{\ot 6}&=W_1\ot W_2\ot W_3\ot W_4\ot W_5\ot W_6\\
&=(W_1\ot W_4) \ot (W_2\ot  W_5)\ot (W_3\ot W_6)\\
&\ \ra S^2W\ot S^2W\ot S^2W\\
&\ \ra S^3(S^2W)
\end{align*}
Note that  $h_{d,n}$ is a linear map, in fact a $GL(W)$-module map.

\exerone{\label{hdnnpow} Show that  
$
h_{d,n}(
x_1^n\cdots x_{d}^n)=
(x_1\cdots x_{d})^n 
$.

Note that the definition of $h_{d,n}$ depends on one's conventions for symmetrization (whether or not to divide by
a constant).  Take the definition of $h_{d,n}$ so that this exercise is true.
}

\exerone{\label{hdualex} Show that $h_{d,n}: S^d(S^nV)\ra S^n(S^dV)$ is \lq\lq self-dual\rq\rq\ in the sense that
$h_{d,n}^T=h_{n,d}: S^n(S^dV^*)\ra S^d(S^nV^*)$. Conclude that $h_{d,n}$ surjective if and only if  $h_{n,d}$ is  injective.}

\begin{proposition} \cite{MR1554881}   $\tker h_{d,n}=I_{d}(Ch_n(W^*))$.
\end{proposition}

\begin{proof}
Say $P=\sum_jx_{1j}^n\cdots x_{d j}^n$.
Let $\ell^1\hd \ell^n\in W^*$.
\begin{align*}
P(\ell^1\cdots \ell^n)&=\langle \ol{P}, (\ell^1\cdots \ell^n)^{d}\rangle\\
&=\sum_j
\langle x_{1j}^n\cdots x_{d j}^n,(\ell^1\cdots \ell^n)^{d}\rangle\\
&=\sum_j
\langle x_{1j}^n,(\ell^1\cdots \ell^n) \rangle 
\cdots \langle x_{d j}^n,(\ell^1\cdots \ell^n) \rangle \\
&=\sum_j\Pi_{s=1}^n\Pi_{i=1}^{d}x_{ij}(\ell_s)
 \\
&=\sum_j
\langle x_{1j}\cdots x_{d j},(\ell^1)^{d} \rangle 
\cdots \langle x_{1j}\cdots x_{d j},(\ell^n)^{d} \rangle  \\
&=
\langle h_{d,n}(P), (\ell^1)^{d}\cdots (\ell^n)^{d}\rangle
\end{align*}
If $h_{d,n}(P)$ is nonzero,   there will be some monomial
it will pair with to be nonzero. On the other hand, if $h_{d,n}(P)=0$,
then $P$ annihilates all points of $Ch_{n}(W^*)$.
\end{proof}

\exerone{\label{fhexa} Show  that if $h_{d,n}: S^d(S^n\BC^m)\ra S^n(S^d\BC^m)$ is not surjective, then $h_{d,n}: S^d(S^n\BC^k)\ra S^n(S^d\BC^k)$
is not surjective for all $k>m$, and that the partitions describing the kernel are the same in both cases if $d\leq m$.}

\exerone{\label{fhexb} Show  that if $h_{d,n}: S^d(S^n\BC^m)\ra S^n(S^d\BC^m)$ is   surjective, then $h_{d,n}: S^d(S^n\BC^k)\ra S^n(S^d\BC^k)$
is  surjective for all $k<m$.} 

\begin{proposition}[Ikenmeyer, Mkrtchyan] \

\begin{enumerate}
\item The kernel of $h_{5,5} :S^5(S^5\BC^5)\ra S^5(S^5\BC^5)$ consists of irreducible modules corresponding to the following partitions:
\begin{align*}\{
&(14,7,2,2), (13,7,2,2,1), (12,7,3,2,1), (12,6,3,2,2),\\
& (12,5,4,3,1), (11,5,4,4,1) ,(10,8,4,2,1) ,(9,7,6,3)\}.
\end{align*}
All these occur with multiplicity one in the kernel, but not all occur with multiplicity one in $S^5(S^5\BC^5)$, so
in particular, the kernel is not an isotypic component.

\item The kernel of $h_{6,6} :S^6(S^6\BC^6)\ra S^6(S^6\BC^6)$ contains, with high probability, a module corresponding to
the partition
$(20,7,6,1,1,1)$.   
\end{enumerate}
\end{proposition}

The phrase \lq\lq with high probability\rq\rq\ means the result was  obtained numerically, not symbolically.

\subsection{Multi-symmetric function formulation}
Given any $GL(V)$-module map $f:U\ra W$, where $U,W$ are modules with support in the root lattice of $GL(V)$, e.g.,
$U,W\subset V^{\ot a\bv}$ for some $a\in \BZ_{>0}$, the injectivity (or surjectivity) of $f$ is equivalent to
the injectivity (or surjectivity) of $f$ restricted to the $\fsl$-weight zero subspaces $f|_0: U_0\ra W_0$, that is
the subspaces of $GL(V)$-weight $(a^{\bv})$. On these subspaces the Weyl group $\FS_{\bv}$ acts, and so the assertion about
a $GL(V)$-module map can be converted to an assertion about an $\FS_{\bv}$-module map, and vice-versa.
This was Briand's approach in \cite{Briand:these}.

\subsection{$\FS_{dn}$-formulation}\label{combsect}
Foulke's conjecture has been well-studied in the combinatorics literature in the following form:
One compares the multiplicities of the $\FS_{dn}$-module induced from the trivial representation
of $\FS_{d}^{\times n}$ with 
the $\FS_{dn}$-module induced from the trivial representation
of $\FS_{n}^{\times d}$. Moreover there is an explicit map between these modules whose kernel in terms
of $\FS_{dn}$-modules  corresponds to the kernel of $h_{d,n}$ as long as the dimension of $V$ is sufficiently large, as this map between
$\FS_{dn}$-modules is just $h_{d,n}$ restricted to the $\fsl$-weight zero subspace.
Some   results, such as Hermite reciprocity  \ref{hermthm}  and  Exercises \ref{fhexa},\ref{fhexb} are less transparent from this perspective and are
the subject
of research articles in combinatorics. However I do not know of a \lq\lq$GL$\rq\rq - proof of the following Theorem  of T. McKay
\cite{MR2394689}:

\begin{theorem}\cite[Thm. 8.1]{MR2394689}
If $h_{d,n}$ is surjective, then $h_{d',n}$ is surjective for all $d'>d$. In other words, if
$h_{n,d}$ is injective, then $h_{n,d'}$ is injective for all $d'>d$.
\end{theorem}

The two statements are equivalent by Exercise \ref{hdualex}.

\subsection{Coordinate ring of the orbit}
Recall from \S\ref{relevantsect},  that if $\tdim W\geq n$, then  $\hat Ch_n(W)=\ol{GL(W)\cdot x_1\cdots x_n}$. Assume $\tdim W=n$, then $G_{x_1\cdots x_n} =T_n^{SL}\rtimes \FS_n=:\G_n$.
By the algebraic Peter-Weyl theorem  \ref{algPW}, 
$$
\BC[GL(W)\cdot (x_1\cdots x_n)]=\bigoplus_{ \ell(\pi)\leq n} (S_{\pi}W^*)^{\oplus \tdim (S_{\pi}W)^{\G_n}},
$$
where $\pi=(p_1\hd p_n)$ is such that $p_1\geq p_2\geq \cdots \geq p_n$ and   $p_j\in \BZ$.
We are only interested in those $\pi$ that are partitions, i.e., where $p_n\geq 0$, as only those could
occur in the coordinate ring of the orbit closure. Define the  $GL$-degree of  a module $S_{\pi}W$ to be
$p_1+\cdots + p_n$ and for a $GL(W)$-module $M$, define $M_{poly}$ to be the sum of the isotypic
components of the $S_{\pi}W$ in $M$ with $\pi$ a partition.
The space of $T^{SL}$ invariants is the  $\fsl(W)$-weight zero space, so we need to compute $(S_{\pi}W)_0^{\FS_n}$.  By Corollary
 \ref{gaycor}(ii) this is $\tmult(S_{\pi}W,S^n(S^sW))$.
If we consider  all the $\pi$'s   together, we conclude
$$
\BC[GL(W)\cdot (x_1\cdots x_n)]_{poly}=\oplus_s S^n(S^sW^*).
$$
In particular, $\oplus_sS^n(S^sW^*)$ inherits  a ring structure.

\subsection{Coordinate ring of the normalization}\label{normalizationsect} In this section I follow \cite{MR1243152}.
There is another variety whose coordinate ring is    as computable as the coordinate ring of the orbit, the normalization
of the Chow variety. We   work in affine space.

An affine variety $Z$ is {\it normal} if $\BC[Z]$ is integrally closed, that is if every element of $\BC(Z)$, the
field of fractions of  $\BC[Z]$,  that is integral over
$\BC[Z]$ (i.e., that satisfies a monic polynomial with coefficients in $\BC[Z]$)  is in $\BC[Z]$. To every affine variety $Z$ 
one may associate a unique normal  affine variety $Nor(Z)$, called the {\it normalization} of $Z$, such that there is
a finite    map $Nor(Z)\ra Z$ (i.e. $\BC[Nor(Z)]$ is integral over $\BC[Z]$) that is generically one to one, in particular it is one to one over the smooth points of $Z$.
For details see \cite[Chap II.5]{MR1328833}.

In particular, there is an inclusion $\BC[Z]\ra \BC[Nor(Z)]$. If the non-normal points of $Z$ form
a finite set, then the cokernel is      finite dimensional. If $Z$ is a $G$-variety,
then $Nor(Z)$ will be too. 

Recall $Ch_n(W)$ is the projection of the Segre variety, but since we want to deal with affine varieties, we will deal with the
cone over it. So instead consider the product map
\begin{align*}
\phi_n: W^{\times n}&\ra S^nW\\
(u_1\hd u_n)&\mapsto u_1\cdots u_n
\end{align*}
Note that i) the image of $\phi_n$ is $\hat Ch_n(W)$, ii) $\phi_n$ is $\G_n=T_W\ltimes \FS_n$ equivariant.

   For any affine algebraic  group $\G$ and any  $\G$-variety $Z$,   one can define the {\it GIT quotient} $Z//\G$ which by definition is the affine
algebraic variety whose coordinate ring is $\BC[Z]^\G$. (When $\G$ is finite, this is just the usual set-theoretic
quotient. In the general case,  $\G$-orbits will be identified in the quotient when  there are no $\G$-invariant regular functions
that can distinguish them.)  If $Z$ is normal, then so is $Z//\G$  (see, e.g.    \cite[Prop 3.1]{MR2004511}).
In our case $W^{\times n}$ is an affine $\G_n$-variety and $\phi_n$ factors through the GIT quotient because it is
$\G_n$-equivariant, so we obtain
a map
$$
\psi_n: W^{\times n}//\G_n\ra S^nW
$$
whose image is still $\hat Ch_n(W)$. Also note that by unique factorization, $\psi_n$ is generically one to one. (Elements of  $W^{\times n}$ 
  of the form $(0,u_2\hd u_n)$ cannot be distinguished from $(0\hd 0)$ by $\G_n$ invariant functions,
so they are identified with $(0\hd 0)$ in the quotient, which is consistent with the fact
that $\phi_n(0,u_2\hd u_n)=0$.)   Observe that $\phi_n$ and $\psi_n$ are $GL(W)=SL(W)\times \BC^*$ equivariant.

Consider the induced map on coordinate rings:
$$
\psi_n^*: \BC[S^nW] \ra \BC[W^{\times n}//\G_n]= \BC[W^{\times n}]^{\G_n}.
$$
Recall that for affine varieties,  $\BC[Y\times Z]=\BC[Y]\ot \BC[Z]$, so
\begin{align*}
\BC[W^{\times n}]&= \BC[W]^{\ot n}\\
&=Sym(W^*)\otc Sym(W^*)\\
&=\bigoplus_{i_1\hd i_n \in \BZ_{\geq 0}} S^{i_1}W^*\otc S^{i_n}W^*.
\end{align*}
Taking torus invariants gives 
$$
\BC[W^{\times n}]^{T_n^{SL}}=\bigoplus_i S^iW^*\otc S^iW^*,
$$
and finally
$$
(\BC[W^{\times n}]^{T_n^{SL}})^{\FS_n}=S^n(S^iW^*).
$$
In summary, 
$$
\psi_n^*: Sym(S^nW^*)\ra \oplus_i(S^n(S^iW^*)), 
$$
and this map respects $GL$-degree, so it gives rise to maps $\tilde h_{d,n}: S^d(S^nW^*)\ra S^n(S^dW^*)$.

\begin{proposition}\label{tildehprop}  $\tilde h_{d,n}=h_{d,n}$.
\end{proposition}
\begin{proof}
   Since elements of the form $x_1^n \cdots x_d^n$ span $S^d(S^nW)$ it will  be sufficient to prove
the maps agree on such elements.  By Exercise \ref{hdnnpow},  $h_{d,n}(x_1^n \cdots x_d^n)=(x_1\cdots x_d)^n$.
On the other hand, in the algebra    $\BC[W]^{\ot n}$, the multiplication is $(f_1\otc f_n)\ocirc (g_1\otc g_n)=f_1g_1\otc f_ng_n$ and this descends to the algebra $(\BC[W]^{\ot n})^{\G_n}$ which is the target
of the algebra map $\psi_n^*$, i.e., 
\begin{align*}
 \tilde h_{d,n}(x_1^n \cdots x_d^n)&=\psi_n^*(x_1^n \cdots x_d^n)\\
 &=\psi_n^*(x_1^n) \ocirc \cdots \ocirc\psi_n^*(x_d^n)\\
 &= x_1^n  \ocirc \cdots \ocirc x_d^n\\
 &=(x_1\cdots x_d)^n.
 \end{align*}
 \end{proof}

\begin{proposition}\label{normalchowprop} $\psi_n: W^{\times n}//\G_n \ra \hat Ch_n(W)$ is the normalization of $\hat Ch_n(W)$.
\end{proposition}

Recall (see, e.g. \cite[p. 61]{MR1328833}) that a regular (see, e.g. \cite[p.27]{MR1328833} for the definition
of regular) map between affine varieties $f: X\ra Y$ such that
$f(X)$ is dense in $Y$ is {\it finite} if $\BC[X]$ is integral over $\BC[Y]$. To prove the proposition, we will need a lemma:

\begin{lemma}\label{acklem} Let $X,Y$ be affine varieties equipped with polynomial $\BC^*$-actions with  unique fixed
points $0_X\in X$, $0_Y\in Y$, and let $f: X\ra Y$  be a $\BC^*$-equivariant morphism such that
as sets, $f\inv(0_Y)=\{ 0_X\}$. Then $f$ is finite.
\end{lemma}

\begin{proof}[Proof of Proposition \ref{normalchowprop}]
Since  $ W^{\times n}//\G_n$ is normal  
and $\psi_n$ is regular and generically one to one,   it just  remains to show $\psi_n$ is finite.

Write $[0]=[0\hd 0]$. To show finiteness, by Lemma \ref{acklem},  it is sufficient to show $\psi_n\inv(0) =[0 ]$ as a set, as $[0 ]$
is the unique $\BC^*$ fixed point in $W^{\times n}//\G_n$, and every $\BC^*$ orbit closure contains $[0 ]$.  
Now $u_1\cdots u_n=0$ if and only if some $u_j=0$, say $u_1=0$. The $T$-orbit closure of $(0,u_2\hd u_n)$ contains
the origin so $[0,u_2\hd u_n]=[0]$.
\end{proof}

\begin{proof}[Proof of Lemma \ref{acklem}]
$\BC[X],\BC[Y]$ are $\BZ_{\geq 0}$-graded, and the hypothesis $f\inv(0_Y)=\{ 0_X\}$ states that $\BC [X]/f^*(\BC[Y]_{>0})\BC[X]$
is a finite dimensional vector space. We want to show that $\BC[X]$ is integral over $\BC[Y]$.
This is a graded version of   Nakayama's Lemma (the algebraic implicit function theorem).
\end{proof}

In more detail (see, e.g. \cite[Lemmas 3.1,3.2]{MR3093509}, or \cite[p136, Ex. 4.6a]{eisenbud}):
\begin{lemma} Let $R,S$ be $\BZ_{\geq 0}$-graded, finitely generated domains over $\BC$ such that $R_0=S_0=\BC$, and
let $f^*: R\ra S$ be an injective  graded algebra homomorphism. If $f\inv (R_{>0})=\{ S_{>0}\}$ as sets, where  
 $f: Spec(S)\ra Spec(R)$ is the induced map on the associated schemes, then $S$ is a finitely generated $R$-module. In particular, it is integral
over $R$.
\end{lemma}
\begin{proof} The hypotheses on the sets says that $S_{>0}$ is the only maximal ideal of $S$ containing the ideal $\fm$ generated
by $f^*(R_{>0})$, so the radical of $\fm$  must equal $S_{>0}$, and in particular $S_{>0}^d$ must be contained in 
it for all $d>d_0$, for some $d_0$. So $S/\fm$ is a finite dimensional vector space, and by the next lemma,
$S$ is  a finitely generated $R$-module.
\end{proof}
\begin{lemma}\label{nextlemma}
Let $S$ be as above, and let $M$ be a $\BZ_{\geq 0}$-graded $S$-module. Assume $M/(S_{>0}\cdot M)$ is a finite dimensional
vector space over $S/S_{>0}\simeq \BC$. Then $M$ is a finitely generated $S$-module.
\end{lemma}
\begin{proof}
Choose a set of homogeneous generators $\{ \ol{x}_1\hd \ol{x}_n\}\subset M/(S_{>0}\cdot M)$ and let $x_j\in M$ be a
homogeneous lift of $\ol{x}_j$. Let $N\subset M$ be the graded $S$-submodule  $Sx_1+\cdots + Sx_n$. Then $M=S_{>0}M+N$, as
let $a\in M$, consider $\ol{a}\in M/(S_{>0}M)$ and lift it to some $b\in N$, so $a-b\in S_{>0}M$, and $a=(a-b)+b$. Now
quotient by $N$ to obtain
\be
\label{coneqn}
S_{>0} \cdot (M/N)= M/N.
\ene
If $M/N\neq 0$, let $d_0$ be the smallest degree such that $(M/N)^{d_0}\neq 0$. But $S_{>0} \cdot (M/N)^{\geq d_0}\subset (M/N)^{\geq d_0+1}$
so there is no way to obtain $(M/N)^{d_0}$ on the right hand side. Contradiction.
\end{proof}

\begin{theorem}\label{brionfhthm}\cite{MR1243152}  For all $n\geq 1$, $\psi_n$ induces a closed immersion
\be\label{brieqn}
(W^{\times n}//\G_n)\backslash [0] \ra S^nW\backslash 0.
\ene
\end{theorem}

\begin{corollary}\cite{MR1243152} \label{brionfhcor} The Hermite-Hadamard-Howe map 
$$
h_{d,n}: S^d(S^nW^*)\ra S^n(S^dW^*)
$$
is surjective for $d$ sufficiently large.
\end{corollary}
\begin{proof}[Proof of Corollary]
Theorem \ref{brionfhthm} implies $(\psi_n^*)_d$ is surjective for $d$ sufficiently large, because
the cokernel of $\psi_n^*$  is supported at a point and thus   must vanish in large degree.
\end{proof}

The proof of Theorem \ref{brionfhthm} will give a second proof  that the kernel of $\psi_n^*$ is
indeed the ideal of $Ch_n(W)$.

\begin{proof}[Proof of Theorem]
Since $\psi_n$ is $\BC^*$-equivariant, we can consider the quotient to projective space
$$
\ul{\psi}_n: ((W^{\times n}//\G_n)\backslash [0])/\BC^* \ra (S^nW\backslash 0)/\BC^* = \BP S^nW
$$
Note that  $((W^{\times n}//\G_n)\backslash [0])/\BC^*$ is $GL(V)$-isomorphic to $(\BP W)^{\times n}/\FS_n$,
as 
$$(W^{\times n}//\G_n)\backslash [0]=(W\backslash 0)^{\times n}/\G_n
$$ 
and $\G_n\times \BC^*= (\BC^*)^{\times n}\rtimes \FS_n$.
So we have
$$
\ul{\psi}_n: (\BP W)^{\times n}/\FS_n \ra   \BP S^nW.
$$
but the projection map   $proj_{\BP S^nW^c}|_{Seg(\BP W\ctimes \BP W)}: \BP W^{\times n}\ra \BP S^nW$  is a closed immersion, and
it   is just averaging over $\FS_n$, i.e., $[w_1\otc w_n]\mapsto [\sum_{\s\in \FS_n}w_{\s(1)}\otc w_{\s(n)}]$
so lifting and quotienting by $\FS_n$ yields \eqref{brieqn}, which is still
a closed immersion. 
\end{proof}

With  more work, in \cite[Thm 3.3]{MR1601139}, Brion obtains an explicit (but enormous) function $d_0(n,\bw)$ 
which is  
\be\label{brionexpl}
d_0(n,\bw)=(n-1)(\bw -1)( (n-1)\left\lfloor \frac{\binom{n+\bw-1}{\bw-1}}{\bw}\right\rfloor  -n)
\ene
for which the $h_{d,n}$ is
surjective for all $d>d_0$ where $\tdim W=\bw$.

\begin{problem} Improve Brion's bound to say,  a polynomial bound in $n$ when $n=\bw$.
\end{problem}

\begin{problem} Note that $\BC[Nor(Ch_n(W))]=\BC[GL(W)\cdot (x_1\cdots x_n)]_{\geq 0}$ and
that the the boundary of the orbit closure is irreducible. Is it true that whenever a $GL(W)$-orbit
closure with reductive stabilizer has an irreducible boundary, that the coordinate ring of
the normalization of the orbit closure equals the positive part of the coordinate ring of the orbit?
\end{problem}

\begin{remark}
 An early  use of geometry in the study of plethysm  was  in \cite{MR1132139} where J. Wahl used his  {\it Gaussian maps}
(local differential geometry)
  to study the decomposition of tensor products of representations of reductive groups.
Then in   \cite{MR1651092}, Manivel used these maps to determine 
  \lq\lq stable\rq\rq\  multiplicities in $S^d(S^nW)$,  where one fixes either $d$ or $n$ and
allows the other to grow.    Brion then developed   more algebraic versions of these
techniques to obtain the results  above. 
\end{remark}

\subsection{The case $\tdim W=2$}
When $\tdim W=2$, every polynomial decomposes as a product of linear factors, so the ideal of $Ch_n(\BC^2)$ is zero. We recover the following theorem of Hermite:
\begin{theorem}[Hermite reciprocity]\label{hermthm}
The map $h_{d,n}: S^d(S^n\BC^2)\ra S^n(S^d\BC^2)$ is an isomorphism for all $d,n$. In particular
  $S^d(S^n\BC^2)$ and $S^n(S^d\BC^2)$ are isomorphic $GL_2$-modules.
\end{theorem}
Often in modern textbooks only the \lq\lq In particular\rq\rq\ is stated.

\subsection{The case $d=n=3$}\label{d3case}

\begin{theorem}[Hadamard \cite{MR1504330}] The map $h_{3,3}: S^3(S^3\BC^n)\ra S^3(S^3\BC^n)$ is an isomorphism.
\end{theorem}
\begin{proof}
Without loss of generality, assume $n= 3$   and $x_1,x_2,x_3$ are independent.
Say we had $P\in I_3(Ch_3(\BC^3))$. Consider $P(\mu (x_1^3+ x_2^3+x_3^3)-\l x_1x_2x_3)$ as a cubic polynomial
on $\pp 1$ with coordinates $[\mu,\l]$. Note that it vanishes at the four points $[0,1],[1, 3],[1,3\o],[1,3\o^2]$
where $\o$ is a primitive third root of unity. Thus it must vanish identically on the $\pp 1$, in particular,
at $[1,0]$, i.e.,  on $x_1^3+x_2^3+x_3^3$. Hence it must
vanish identically on $\s_3(v_3(\pp 2))$. But $\s_3(v_3(\pp 2))\subset \BP S^3\BC^3$ is a hypersurface of
degree four. A cubic polynomial vanishing on a hypersurface of degree four is  identically zero.
\end{proof}

\begin{remark}
The above proof is due to A. Abdesselam (personal communication). It is a variant of Hadamard's original proof, where
instead of $x_1^3+ x_2^3+x_3^3$ one uses an arbitrary cubic $f$, and generalizing $x_1x_2x_3$ one uses the Hessian
$H(f)$. Then the curves $f=0$ and $H(f)=0$ intersect in $9$ points (the nine flexes of $f=0$) and there are four
groups of three lines going through these points, i.e. four places where the polynomial becomes a product of linear forms.
\end{remark}

\subsection{The Chow variety and a conjecture in combinatorics}\label{chowcombin}
From Exercise \ref{snssv}, the  trivial $SL_n$-module $S_{n^n}\BC^n$ occurs in $S^n(S^n\BC^n)$ with multiplicity one when $n$ is even and zero when $n$ is odd.

\begin{conjecture}[Kumar \cite{kumarcoordring}]\label{kumarconj} Let $n$ be even, then for all $i\leq n$,  $S_{n^i}\BC^n\subset  \BC[Ch_n(\BC^n)]$.
\end{conjecture}

It is not hard to see that the $i=n$ case implies the others.
Adopt the notation that if $\pi=(p_1\hd p_k)$, then $m\pi=(mp_1\hd mp_k)$.
By taking Cartan products in the coordinate ring, the conjecture would imply:  

\begin{conjecture}[Kumar  \cite{kumarcoordring}] For all partitions $\pi$ with $\ell(\pi)\leq n$, the module $S_{n\pi}\BC^n$ occurs in $\BC[Ch_n(\BC^n)]$.
In particular, $S_{n\pi}\BC^{n^2}$ occurs in $\BC[\Det_n]$ and $\BC[\Perm^n_n]$.
\end{conjecture}

Conjecture \ref{kumarconj} turns out to be related to a famous conjecture in combinatorics:
an $n\times n$ matrix such that each row and column consists of the integers $\{ 1\hd n\}$ is called a {\it Latin square}.
To each row and column one can associate an element $\s\in \FS_n$ based on the order the integers appear. 
Call the products of all the signs of these permutations the sign of the Latin square.

\begin{conjecture}[Alon-Tarsi  \cite{MR1179249}] \label{ATconj} Let $n$ be even. The number of sign $-1$ Latin squares of size $n$ is not equal
to the number of sign $+1$ Latin squares of size $n$.
\end{conjecture}

In joint work, Kumar and I have shown:

\begin{proposition}\label{kconjequiv} Fix $n$ even. The following are equivalent:
\begin{enumerate}
\item The Alon-Tarsi conjecture for $n$.
\item Conjecture \ref{kumarconj} for $n$  with $i=n$.
\item   $\int_{g\in SU(n)} (\tperm_n(g))^n d\mu\neq 0$, where $d\mu$ is Haar measure.
\item Let $\BC^{n^2}$ have coordinates $x^i_j$ and the dual space coordinates $y^i_j$, then
$$
\langle (\tperm_n(y))^n,(\tdet_n(x))^n\rangle \neq 0
$$
which may be thought of as a pairing between homogeneous polynomials of degree $n^2$ and homogeneous differential operators
of order $n^2$.
\end{enumerate}
The following two statements are equivalent and  would imply the above are true:

(i)   $\int_{g\in SU(n)} \Pi_{1\leq i,j\leq n}g^i_j d\mu\neq 0$, where $d\mu$ is Haar measure.

(ii)  $\langle \Pi_{ij}y^i_j, \tdet_n(x)^n\rangle\neq 0$.
\end{proposition}

Currently the Alon-Tarsi conjecture is known to be true for $n=p\pm 1$, where $p$ is a prime number
\cite{MR2646093,MR1451417}.

To see the equivalence of (1) and (2), in \cite{MR1271866} they showed that the Latin square conjecture is true
for even $n$ if and only if the \lq\lq column sign\rq\rq\ Latin square conjecture holds, where one instead computes
the   the products of the signs of the permutations of the columns. Then  expression \eqref{polyexpr}
gives the equivalence. The equivalence of (3) and (4) comes from the Peter-Weyl theorem and the equivalence
of (2) and (3) from the fact that one can restrict to a maximal compact, and integration over the group picks
out the trivial modules.

\begin{problem} Find explicit modules that either are or are not in the kernel of the Hermite-Hadamard-Howe map.
For example any module with at most two parts is clearly not in the kernel.  \end{problem}

\section{Secant varieties of the Chow variety and depth three circuits}\label{depth3sect}

Recently there has been substantial progress regarding shallow circuits.  I first
define a circuit, which is the model of computation generally used in algebraic
complexity theory, and then I   describe the varieties associated to shallow circuits as well as    recent results
and  conjectures regarding shallow circuits in geometric language.

\begin{definition}\index{arithmetic circuit}\label{arithcirdef}
An {\it arithmetic circuit} $\cC$ is a finite, acyclic, directed graph with vertices of
in-degree $0$ or $2$ and exactly one vertex of out-degree $0$.
The vertices of in-degree $0$  are labeled by elements of $\BC\cup\{ x_1\hd x_n\}$, and called {\it inputs}.
Those of in-degree $2$  are labeled
with $+$ or $*$ and  are called {\it   gates}. If the out-degree of $v$ is $0$, then $v$ is called an {\it output gate}.
 The {\it size} of $\cC$ is the number of edges.  
From a circuit $\cC$, one can construct a polynomial $p_{\cC}$ in the variables $x_1\hd x_n$.
\end{definition}

 \begin{figure}[!htb]\begin{center}
\includegraphics[scale=.4]{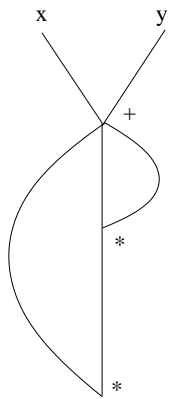}
\caption{\small{Circuit for $(x+y)^3$}}  
\end{center}
\end{figure}

\exerone{Show that if one instead uses the number of gates to define the size, the asymptotic size estimates are the same.
(Size is    sometimes defined as the number of gates.)}

\medskip 

To each vertex $v$ of a circuit $\cC$ we  associate  the  polynomial that is computed at $v$, which will be denoted $\cC_v$.
In particular the polynomial associated with the output gate is the polynomial computed by $\cC$. 
The {\it depth} of $\cC$ is the length  of (i.e., the  number of edges in)  the longest path in $\cC$  from an input to an output.
If a circuit has small depth, the polynomial it computes can be computed quickly in parallel.

The {\it formula size} of $f$ is the smallest tree circuit computing $f$. Tree circuits are  
called {\it formulas}.

 Circuits of bounded
depth (called {\it shallow circuits}) are used to study the complexity of calculations done in parallel. When one studies
circuits of bounded depth, one must allow gates to have an arbitrary number of edges coming in to them (\lq\lq unbounded fanin\rq\rq ).
For such circuits, multiplication by constants is considered \lq\lq free.\rq\rq  

There is a substantial literature dedicated to showing that given any circuit computing a polynomial, there is a \lq\lq slightly larger\rq\rq\ 
shallow circuit that computes the same polynomial. Recently there have been significant advances for circuits of depths $3$
\cite{DBLP:journals/eccc/GuptaKKS13}  and $4$ \cite{tavenas,koirand4,AgrawalVinay}
and a special class of circuits of depth $5$ \cite{DBLP:journals/eccc/GuptaKKS13}. The circuits of bounded depth that are trees
have a nice variety associated to them which I now describe. In the literature they deal with inhomogeneous circuits,
but, as I describe below (following a suggestion of K. Efremenko), this can be avoided, so we will deal exclusively
with homogeneous circuits, that is, those computing homogeneous polynomials at each step along the way.

Following \cite{MR2733384}, for varieties $X\subset \BP S^aW$ and $Y\subset \BP S^bW$, defined the {\it multiplicative
join} of $X$ and $Y$, $MJ(X,Y):=\{ [xy]\mid [x]\in X,\ [y]\in Y\}\subset \BP S^{a+b}W$, and define $MJ(X_1\hd X_k)$ similarly.
Let $\mu_k(X)= MJ(X_1\hd X_k)$ when all the $X_j=X$, which is a multiplicative analog of the secant variety.
Note that $\mu_k(\BP W)=Ch_k(W)$.   
The varieties associated to the polynomials computable by bounded depth formulas are of the form
$\s_{r_k}(\mu_{d_{k-1}}(\s_{r_{k-2}}(\cdots \mu_{d_1}(\BP W)\cdots )))$, and 
$\mu_{d_{k+1}}(\s_{r_k}(\mu_{d_{k-1}}(\s_{r_{k-2}}(\cdots \mu_{d_1}(\BP W) \cdots ))))$.

\begin{remark} For those interested in circuits, note that if the first level consists of addition gates,
this is \lq\lq free\rq\rq\ from the perspective of algebraic geometry, as since we are not choosing coordinates,
linear combinations of basis vectors are not counted. More on this below.
\end{remark}

Useful depth three circuits are always trees   where the first level consists of additions,
the second multiplications, and the third an addition that adds all the outputs of the
second level together. Such are called $\Sigma\Pi\Sigma$ circuits.

A circuit is {\it homogeneous} if the polynomial produced by each gate is homogeneous, and otherwise it is {\it inhomogeneous}.
The relation between secant varieties of Chow varieties and depth three circuits is as follows:

\begin{proposition}\label{chcirprop} A polynomial $P\in S^nW$   in $\s_r^0(Ch_n(W))$   is computable
by a homogeneous circuit of size $r+nr(1+\bw)$. If $P\not\in \s_r^0(Ch_n(W))$, then
$P$ cannot be computed by a homogeneous circuit of size $n(r+1)+(r+1)$.
\end{proposition}
\begin{proof} In the first case, $P=\sum_{j=1}^r(x^1_j\cdots x^n_j)$ for some $x^i_j\in W$.
Expressed in terms of a fixed basis of $W$, each $x^i_j$ is a linear combination of at worst $\bw$
basis vectors, thus to create each one requires at worst $nr\bw$ additions. Then to multiply them
in groups of $n$ is $nr$ multiplications, and finally to add these together is $r$ further additions.
In the second case, at best $P$ is in $\s_{r+1}^0(Ch_n(W))$, in which case, even if each 
of the $x^i_j$'s is a basis vector (so no initial additions are needed), we still must
perform $n(r+1)$ multiplications and $r+1$ additions.
\end{proof}

I first explain why the computer science literature generally allows inhomogeneous depth three circuits, and then why one does
not need to do so.

\medskip

\subsection{Why   homogeneous depth three circuits do not appear useful at first glance}
Using the flattening  (see \S\ref{relevantsect}),
$(\tdet_n)_{\lceil \frac n2\rceil,\lfloor \frac n2\rfloor}: S^{\lceil \frac n2\rceil}W\ra
S^{\lfloor \frac n2\rfloor}W$  and writing $W=E\ot F=\BC^n\ot \BC^n$, the image of this map 
is easily seen to be $\La {\lfloor \frac n2\rfloor}E\ot \La {\lfloor \frac n2\rfloor}F$, the minors of
size $\lfloor \frac n2\rfloor$. For the permanent    one similarly gets sub-permanents. Thus
$$
\ur_S (\tdet_n)\geq {\binom n{\lfloor \frac n2\rfloor} }^2,\ \ \ur_S (\tperm_n)\geq {\binom n{\lfloor \frac n2\rfloor} }^2.
$$
Recalling  that $\binom{2m}m \sim \frac {4^m}{\sqrt{\pi m}}$,
we have $[\tdet_n],[\tperm_n]\not\in\s_{O( \frac{4^n}{ n }  )}v_n(\BP W)$. 

In  \cite{MR2842085} they showed
\be\label{chowrk}
\bold R_S(x_1\cdots x_n)=2^{n-1}.
\ene
The upper bound follows from the expression
\be\label{fischer}
 x_1\cdots x_n= \frac{1}{2^{n-1} n!} \sum_{\epsilon \in \{-1,1\}^{n-1}} 
(x_1 + \epsilon_1 x_2 + \dots + \epsilon_{n-1} x_n)^n \epsilon_1 \cdots \epsilon_{n-1} ,
\ene
a sum with $2^{n-1}$ terms.
(This expression    dates at least back to  \cite{MR1573008}.) 
In particular 
$$\s_r(Ch_n(W))\subset \s_{r2^n}(v_n(\BP W)).
$$
We conclude, for any constant $C$ and $n$ sufficiently large, that 
$$ \tdet_n\not\in \s_{C\frac{2^n}n}(Ch_n(W)),
$$
and similarly for the permanent. 
By Proposition \ref{chcirprop}, we conclude:

\begin{proposition} \cite{MR1486927} The polynomial sequences  $\tdet_n$ and $\tperm_n$ do not admit
depth three circuits of size $2^n$.
\end{proposition}

(In \cite{MR1486927} they consider all partial derivatives
of all orders simultaneously, but the bulk of the dimension is concentrated in the middle order flattening, so
one does not gain very much this way.)
Thus homogeneous depth three circuits at first sight do not seem that powerful because a polynomial sized 
homogeneous depth $3$ circuit  cannot compute the
determinant.

To make matters worse, consider the polynomial corresponding to 
iterated matrix multiplication of three by three matrices   $IMM^3_k\in S^k(\BC^{9k})$. It  is complete for $\vp_e$, polynomials with small formula sizes
(see \S\ref{complexapp}), and 
also has an exponential  lower bound for its  Chow border rank.

\exerone{\label{immexer}Use flattenings to show $\ur_S(IMM^3_k)\geq (const.) 3^k$,   and conclude $IMM^3_k\not\in \sigma_{poly(k)}(Ch_k(W))$.}

By Exercise \ref{immexer}, homogeneous depth three circuits (na\"\i vely applied) cannot even capture
sequences of polynomials admitting small formulas.

Another benchmark in complexity theory are the elementary symmetric functions
$$
e^k_n:=\sum_{I\subset [n], |I|=k} x_{i_1}\cdots x_{i_k}.
$$
To fix ideas, set 
  $n=4k$. Let $k=2p$.  Consider the flattening:
$$
(e^k_{4k})_{p,p}: S^p\BC^{2k *}\ra S^p\BC^{2k}
$$
It has image all monomials $x_{i_1}\cdots x_{i_p}$ with the $i_j$ distinct, so its rank is
$\binom{4k}{\frac k2}$ and since 
$
\binom{4k}
{\frac k2 }
/
\binom{ k}{\frac k2 }
$
grows faster than any polynomial in $k$, we conclude even the elementary symmetric function $e^k_{4k}$
cannot be computed by a homogeneous depth three circuit of polynomial size. This last assertion
is \cite[Thm. 0]{MR1486927}, where they show more generally (by the same method) that
$e^{2d}_{n}\not\in \s_{\Omega((\frac n{4d})^d)}^0(Ch_{2d}(\BC^n))$.

 \begin{remark} Strassen \cite{MR0395147} proved a lower bound of $\Omega(n\tlog n)$ for the size of {\it any}
 arithmetic circuit computing  all the $e^{j}_n$ simultaneously. 
 \end{remark}
\subsection{Upper bounds for homogeneous depth three circuits}

The most famous homogeneous depth three circuit is probably Ryser's formula for the permanent:
\be\label{ryserformula}
  \tperm_n = 2^{-n+1} \sum_{\substack{\epsilon \in \{-1,1\}^n \\ \epsilon_1=1}}
    \prod_{1 \leq i \leq n} \sum_{1 \leq j \leq n} \epsilon_i \epsilon_j x_{i,j} ,
\ene
the outer sum is taken over $n$-tuples $\ep=(\epsilon_1=1, \epsilon_2,\dots,\epsilon_n)$.
Note that each term in the outer sum is a product of $n$ independent linear forms
and there are $2^{n-1}$ terms.
In particular  $[\tperm_n]\in \s_{2^{n-1}}^0(Ch_n(\BC^{n^2}))$,
and since 
 $Ch_n(\BC^{n^2})\subset \s_{2^{n-1}}^0(v_n(\pp{n^2-1}))$, we obtain
$\bold R_S(\tperm_n)\leq 4^{n-1}$.


\subsection{Homogeneous depth three circuits for padded polynomials}
At first glance it seems polynomial sized depth 3 circuits are useless, as they cannot compute even simple sequences
of polynomials as we just saw. However,   if one allows padded polynomials,  the situation
changes dramatically. (As mentioned above, in \cite{DBLP:journals/eccc/GuptaKKS13} and elsewhere
  they consider inhomogeneous polynomials and circuits  instead of padding.)
The following geometric version of a result of Ben-Or (presented below as a Corollary) was suggested
by K.  Efremenko:

\begin{proposition} Let $\BC^{m+1}$ have coordinates $\ell, x_1\hd x_m$ and let
$e^k_m=e^k_m(x_1\hd x_m)$.    For all $k\leq m$, 
$\ell^{m-k}e^k_m  \in \s_{m}^0(Ch_m(\BC^{m+1}))$.
\end{proposition}
\begin{proof}
Fix an integer $u\in \BZ$ and define
\begin{align*}
g_u(x,\ell)&=\prod_{i=1}^{m}(x_i+u\ell)\\
&=\sum_k u^{m-k}e^k_m(x)\ell^{m-k}
\end{align*}
Note $g_u(x,\ell)\in Ch_m(\BC^{m+1})$. 
Letting $u=1\hd m$, we may use the 
inverse of the Vandermonde
matrix to write each  $\ell^{m-k}e^k_m$ as a sum of $m$ points in  $Ch_m(\BC^{m+1})$ because
$$
\begin{pmatrix}
1^0&1^1&\cdots& 1^m\\
2^0&2^1&\cdots& 2^m\\
&   \vdots & &\\
m^0&m^1&\cdots &m^m
\end{pmatrix}
\begin{pmatrix} \ell^{m-1}e^1_m\\ \ell^{m-2}e^2_m \\ \vdots \\ \ell^{0}e^m_m\end{pmatrix}
=
\begin{pmatrix} g_1(x,\ell)\\ g_2(x,\ell) \\ \vdots \\ g_m(x, \ell)\end{pmatrix}.
 $$
\end{proof}

\begin{corollary}[Ben-Or]\label{benorthm} $\ell^{m-k}e^k_m$ can be computed by a homogeneous depth three circuit of size $3m^2+m$.
\end{corollary}
\begin{proof} As remarked above, for any point of $\s_{r}Ch_n(\BC^{m+1})$ one   gets
a circuit of size at most  $r+nr+rn(m+1)$, but here at the first level all the addition gates have fanin two 
(i.e., there are two inputs to each addition gate) instead of the
possible $m+1$. 
\end{proof}

\begin{problem}[\cite{MR2901512} Open problem 11.1] Find an explicit sequence of polynomials $P_m\in S^m\BC^{\bw-1}$
such that $\ell^{n-m}P_m\not\in \s_r(Ch_n(W))$, whenever  $r,\bw,n$ are  polynomials in $m$ and $m$ is sufficiently large.
\end{problem} 

\begin{remark}
The best lower bound   for computing the  $e^k_n$ via a $\Sigma\Pi\Sigma$ circuit  is $\Omega(n^2)$ \cite{MR1867306},
so Corollary \ref{benorthm} is very close to (and may well be) sharp.
\end{remark}

\subsection{Depth reduction}\label{depthred}

The following theorem combines results of \cite{MR0660280,DBLP:journals/eccc/GuptaKKS13,tavenas,koirand4,AgrawalVinay}
as explained in the discussion below. (The circuit bounds stated in the theorem come from \cite{tavenas}.)
A  $\Sigma\Lambda\Sigma\Lambda\Sigma$ circuit is a depth $5$ circuit where the first level consists of additions,
the second of \lq\lq powering gates\rq\rq, where a powering gate takes  $f$ to $f^{\d}$ for some $\d$ (the size of the circuit
takes the size of $\d$ into account), the third additions, the fourth powering gates
and the fifth an addition. See \cite{DBLP:journals/eccc/GuptaKKS13} for more details. 
The  $\Sigma\Lambda\Sigma\Lambda\Sigma$ circuits are related to the variety $\s_{r_1}(v_{\frac d\d}(\s_{r_2}(v_{\d}(\BP V)))\subset \BP S^{d}V$
in the same way that the $\Sigma\Pi\Sigma$ circuits are related to $\s_r(Ch_n(V))$.

\begin{theorem} Let $d=n^{O(1)}$ and let $P\in S^d\BC^n$  be a polynomial  that can be computed by
a circuit of size $s$.

Then:
\begin{enumerate}

\item   $f$ is   computable by a homogeneous $\Sigma\Pi\Sigma\Pi$ circuit of   size $2^{O(\sqrt{d \tlog(ds)\tlog(n)})}$.

\item    $f$ is  computable by a $\Sigma\Pi\Sigma$ circuit of size $2^{O(\sqrt{d\tlog(n)\tlog(ds)})}$. In particular, 
$[\ell^{N-d}P]\in \s_r(Ch_{N}(\BC^{n+1}))$ with $rN=2^{O(\sqrt{d\tlog(n)\tlog(ds)})}$.

\item  $f$ is computable,  for some $\d\simeq \sqrt{d}$,  by a homogeneous  $\Sigma\Lambda\Sigma\Lambda\Sigma$ circuit of size 
$  2^{O(\sqrt{d \tlog(ds)\tlog(n)})} 
$.
 In particular, 
  $[P]\in \s_{r_1}(v_{\frac d \d}(\s_{r_2}(v_{\d}(\pp{n-1}))))$ with $r_1r_2 (\d+1)= 2^{O(\sqrt{d \tlog(ds)\tlog(n)})} $.
\end{enumerate}
\end{theorem}

The   \lq\lq in particular\rq\rq\ of (2) follows by setting the circuit size  equal to  $r+Nr$ (the smallest, i.e., worst case size of a circuit 
for a point of $\s_r(Ch_N(\BC^{n+1}))$ that is not in a smaller variety). The 
 \lq\lq in particular\rq\rq\ of (3) 
follows similarly, as the smallest  circuit for a point of $\s_{r_1}(v_{d-\d}(\s_{r_2}(v_{\d}(\pp{n-1}))))$
not in a smaller variety is $r_1r_2 (\d+1)+\frac d{\d}r_1$.   

\begin{corollary}\cite{DBLP:journals/eccc/GuptaKKS13} $[\ell^{n-m}\tdet_m]\in \s_r(Ch_n(\BC^{m^2+1}))$ where $rn=2^{O(\sqrt{m}\tlog m)}$.
\end{corollary}

\begin{proof} The determinant admits a circuit of size $m^4$, so it admits a $\Sigma\Pi\Sigma$ circuit
of size 
$$2^{O(\sqrt{m\tlog(m)\tlog(m*m^4)})}\sim 2^{O(\sqrt{m}\tlog m)},
$$
so its padded version 
lies in $\s_r(Ch_n(\BC^{m^2+1}))$ where $rn=2^{O(\sqrt{m}\tlog m)}$.\end{proof}

\begin{corollary}\cite{DBLP:journals/eccc/GuptaKKS13} 
If for all but finitely many   $m$ and all    $r,n$ with  $rn=2^{ \sqrt{m}\tlog (m)\o(1)}$, one has 
  $[\ell^{n-m}\tperm_m]\not\in \s_r(Ch_n(\BC^{m^2+1}))$,
then there is no  circuit of polynomial size computing the permanent, i.e., $\vp\neq \vnp$.
\end{corollary}

\begin{proof} In this case the $s$ in (2) cannot be a polynomial.
\end{proof}

\begin{corollary} \cite{DBLP:journals/eccc/GuptaKKS13}  If  for all but finitely  many $m$, $\d\simeq \sqrt{m}$, 
and all   $r_1,r_2$ such that  $r_1r_2=2^{ \sqrt{m} \tlog (m)\o(1)}$, one has
 $[\tperm_m]\not\in \s_{r_1}(v_{m/\d}(\s_{r_2}(v_{\d}(\pp{m^2-1}))))$,  
then there is no  circuit of polynomial size computing the permanent, i.e., $\vp\neq \vnp$.
\end{corollary}

\begin{proof} In this case the $s$ in (3) cannot be a polynomial.
\end{proof}
 
These Corollaries give rise to Conjectures \ref{chowvnp} and \ref{slsvnp} stated in \S\ref{valprobs}.

The results above follow from an extensive amount of research. Here is an overview:



In \cite{DBLP:journals/eccc/GuptaKKS13} they prove their upper bounds for  the size of an inhomogeneous depth three
circuit computing a polynomial,  in terms of the size of an arbitrary circuit
computing the polynomial,  by first applying  the work of \cite{koirand4,AgrawalVinay}, which allows one to
reduce an arbitrary circuit of size $s$ computing a polynomial of degree $d$ in $n$ variables   to a formula of size 
$2^{O(\tlog s\tlog d)}$  and   depth $d$.
Next they reduce  to a  depth four circuit of size $s'=2^{O(\sqrt{d\tlog d\tlog s\tlog n})}$. This second passage is via
iterated matrix multiplication. From the depth four circuit, they use \eqref{fischer} to convert all multiplication
gates to sums of elements of the Veronese (what they call $\Sigma\Lambda\Sigma$ circuits), to have a depth
five circuit of size  $O(s')$ and of the form $\Sigma\Lambda \Sigma\Lambda \Sigma$. Finally, they use Newton's identities
to convert power sums to elementary symmetric functions which keeps the size at $O(s')$ and drops the depth to three.

\begin{remark} In \cite{DBLP:journals/eccc/GuptaKKS13}, they also show that, for a similar price, 
one can convert a depth three circuit to a $\Sigma\Lambda\Sigma \Lambda \Sigma$
 circuit by using the inverse identities without substantially increasing the size.
\end{remark}

\begin{remark} Ultimately, if one wants to separate $\vp_{ws}$  from $\vnp$, one will have to find polynomials
that separate $\tdet_n$ from $\ell^{n-m}\tperm_m$. These auxiliary varieties arising from shallow circuits should
be viewed as a guide to how to look for such equations, not as a way to avoid finding them.
\end{remark}

   \begin{remark}
 Note the expected dimension of $\s_r(Ch_d(W))$ is $rd\bw+r-1$. If we take
 $d'=d2^m$ and work instead with padded polynomials $\ell^{2^m}P$, the expected dimension of $\s_r(Ch_{d'}(W))$ is $2^mrd\bw+r-1$.
 In contrast, the expected dimension of $\s_r(v_{d-a}(\s_{\rho}(v_a(\BP W))))$ does not change when
 one increases the degree, which gives some insight as to why padding is so useful for homogeneous depth three circuits
 but not for $\Sigma\Lambda\Sigma \Lambda \Sigma$ circuits.
\end{remark}

\section{Non-normality}\label{kumarpfsect}
I follow \cite{MR3093509} in this section.
Throughout this section I make the following assumptions and adopt the following notation:
\be\label{assume}{\rm{\bold{Assumptions}}:} 
\ene 
\begin{enumerate}
\item $V$ is a $GL(W)$-module,
\item $P\in V$ is such that the $SL(W)$-orbit of $P$ is closed.
\item Let $\cP^0 := GL(W)\cdot P $ and $\cP:=\ol{GL(W)\cdot P}\subset V$ denote its orbit and orbit closure, and let $\partial \cP=\cP\backslash \cP^0$
denote its boundary, which we assume to be more than zero  (otherwise $[\cP]$ is homogeneous). 
\item Assume the stabilizer $G_P\subset GL(W)$ is reductive,
 which is equivalent   
 (by a theorem of Matsushima \cite{MR0109854}) to requiring  that  $\cP^0$ is an affine variety. 
 \end{enumerate}

This situation holds when $V=S^nW$, $\tdim W=n^2$ and $P=\tdet_n$ or $\tperm_n$ as well as when $\tdim W=rn$ and
$P=S^r_n:=\sum_{j=1}^rx_1^j\cdots x_n^j$, the \lq\lq sum-product polynomial\rq\rq , in which case $\cP=\hat \s_r(Ch_n(W))$.

\begin{lemma}\label{kumlem1}  \cite{MR3093509} Assumptions as in \eqref{assume}. Let $M\subset \BC[\cP]$ be a nonzero $GL(W)$-module,  and let $Z(M)=\{ y\in \cP \mid f(y)=0\ \forall f\in M\}$ denote its zero set.
Then   $0\subseteq Z(M)\subseteq \partial\cP$.

If moreover  $M\subset  I(\partial\cP)$, 
then as sets, $Z(M) = \partial \cP$.
\end{lemma}
\begin{proof}
Since $Z(M)$ is a $GL(W)$-stable subset, if it contains a point of $\cP^0$ it must contain all of $\cP^0$ and thus $M$
vanishes identically on $\cP$, which cannot happen as $M$ is nonzero.   
Thus $Z(M)\subseteq \partial \cP$.
For the second assertion, since $ M\subset I(\partial \cP)$, we also have $Z(M)\supseteq \partial\cP$.
\end{proof}

\begin{proposition}\label{kumprop1}  \cite{MR3093509}  Assumptions as in \eqref{assume}.   The space of $SL(W)$-invariants of positive degree in the coordinate ring of $\cP$, 
$\BC[\cP]^{SL(W)}_{>0}$,  is non-empty and contained in $I(\partial \cP)$. Moreover, 
\begin{enumerate}
\item  any element of
$\BC[\cP]^{SL(W)}_{>0}$ cuts out $\partial \cP$ set-theoretically, and 
\item  the components of $\partial\cP$
all have codimension one in $\cP$.
\end{enumerate}
\end{proposition}
\begin{proof} 
To study $\BC[\cP]^{SL(W)}$, consider the GIT quotient $\cP// SL(W)$ whose coordinate ring, by definition, is $\BC[\cP]^{SL(W)}$.
It parametrizes the closed $SL(W)$-orbits in $\cP$, so it is non-empty. Thus $\BC[\cP]^{SL(W)}$ is nontrivial.
 
We will show
  that  every $SL(W)$-orbit in $\partial P$ contains $\{ 0\}$ in its closure, i.e., that $\partial\cP$ maps to zero in the GIT quotient.
This will imply   any $SL(W)$-invariant of positive degree is in $I(\partial \cP)$ because  any non-constant function on the GIT quotient vanishes
on the inverse image of $[0]$. 
Then (1) follows from  Lemma \ref{kumlem1}. The zero set of a single polynomial,
if it is not empty, has codimension one, which  implies the components of $\partial \cP$ are all of  
codimension one, proving (2).

It remains to show $\partial\cP$ maps to zero in $\cP// SL(W)$, where $\rho: GL(W)\ra GL(V)$ is the representation. This GIT quotient  
 inherits a $\BC^*$ action via $\rho(\l Id)$, for $\l\in \BC^*$.
Its normalization is just the affine line $\BA^1=\BC$. To see this, consider
the $\BC^*$-equivariant map $\s : \BC\ra \cP$ given by $z\mapsto \rho(zId)\cdot P$, which  descends to a map $\ol{\s}: \BC\ra \cP// SL(W)$.
Since the $SL(W)$-orbit of $P$  is closed, for any $\l\in \BC^*$,  $\rho(\l Id)P$ does not map  to zero
in the GIT quotient, so  we have $\ol{\s}\inv([0])=\{0\}$ as a set. Lemma \ref{acklem} applies so 
$\ol{\s}$ is finite and gives the normalization. 
Finally,  
  were there a closed nonzero orbit in $\partial \cP$, it would have to equal $SL(W)\cdot \s(\l)$ for some
$\l\in \BC^*$ since $\ol{\s}$ is surjective. But $SL(W)\cdot \s(\l)\subset \cP^0$.
\end{proof}

\begin{remark} That each irreducible
component of  $\partial \cP$ is of   codimension one in $\cP$  is due to Matsushima \cite{MR0109854}. It is   a consequence
of his result mentioned above.
\end{remark}

The key to proving non-normality of $\hat\Det_n$ and $\hat\Perm^n_n$ is to find an $SL(W)$-invariant in the coordinate ring
of the normalization (which has a $GL(W)$-grading), which does not occur in the corresponding
graded component of  the coordinate ring of $  S^nW$, so it cannot
occur in the coordinate ring of any $GL(W)$-subvariety.

\begin{lemma}\label{nnorlem}   Assumptions as in \eqref{assume}.   Let $P\in S^nW$ be such that $SL(W)\cdot P$ is closed and $G_P$ is reductive.
Let $d$ be the smallest positive $GL(W)$-degree such that $\BC[\cP^0]^{SL(W)}_d\neq 0$.
If $n$ is even and $d<n\bw$ (resp. $n$ is odd and $d<2n\bw$) then $\cP$ is not normal.
\end{lemma}
\begin{proof}
Since $\cP^0\subset \cP$ is a Zariski open subset, we
 have the equality of $GL(W)$-modules $\BC(\cP)=\BC(\cP^0)$.
By restriction of functions
  $\BC[\cP]\subset \BC[\cP^0]$   and thus
$\BC[\cP]^{SL(W)}\subset \BC[\cP^0]^{SL(W)}$.  
Now $\cP^0//SL(W)=\cP^0/SL(W)\simeq \BC^*$, so $\BC[\cP^0]^{SL(W)}\simeq \oplus_{k\in \BZ} \BC\{ z^k\}$.
Under this identification, $z$ has $GL(W)$-degree $d$. 
By Proposition \ref{kumprop1}, $\BC[\cP]^{SL(W)}\neq 0$. Let $h\in \BC[\cP]^{SL(W)}$ be the smallest element in positive degree.
Then  $h=z^k$ for some $k$. Were $\cP$ normal, we would   have $k=1$.

But now we also have a surjection $\BC[S^nW]\ra \BC[\cP]$, and by Exercise \ref{snssv} the 
smallest possible  $GL(W)$-degree of an $SL(W)$-invariant in $\BC[S^nW]$ when $n$ is even (resp. odd) 
is $\bw n$ (resp. $2\bw n$)  which would occur  in $S^{\bw}(S^nW)$ (resp. $S^{2\bw}(S^nW)$). We obtain a contradiction.
\end{proof}

\begin{theorem}[Kumar \cite{MR3093509}]\label{kumarnorthm}
For all $n\geq 3$, $\Det_n$ and $\Perm^n_n$ are not normal. For all $n\geq 2m$ (the range of interest), $\Perm^m_n$ is not normal.
\end{theorem}

I give the proof for $\Det_n$, the case of   $\Perm^n_n$ is   an easy exercise.
Despite
the variety being much more singular, the proof for $\Perm^m_n$ with $m>n$  is more difficult,  see \cite{MR3093509}.

\begin{proof}
We will show that 
when $n$ is congruent to $0$ or $1$ mod $4$, $\BC[\Det_n^0]^{SL(W)}_{n-GL}\neq 0$ and 
when $n$ is congruent to $2$ or $3$ mod $4$, $\BC[\Det_n^0]^{SL(W)}_{2n-GL}\neq 0$. Since $n,2n<(n^2)n$
Lemma \ref{nnorlem} applies.

The $SL(W)$-trivial modules are $(\La {n^2}W)^{\ot s}=S_{s^{n^2}}W$. 
Write $W=E\ot F$.
We want to determine the lowest degree trivial $SL(W)$-module  that has a 
$G_{det_n}=(SL(E)\times SL(F) /\mu_n)\rtimes \BZ_2$ invariant.
 We have the decomposition $(\La {n^2}W)^{\ot s}=(\op_{|\pi|=n^2} S_{\pi}E\ot S_{\pi'}F)^{\ot s}$, where $\pi'$ is the conjugate partition to $\pi$.
Thus   $(\La {n^2}W)^{\ot s}$ contains the  trivial $SL(E)\times SL(F)$ module  $(\La nE)^{\ot ns}\ot (\La nF)^{\ot ns}$
with  multiplicity one. (In the language of \S\ref{glwreps}, $k_{s^{n^2},(sn)^n,(sn)^n}=1$.)   Now we consider the effect of the $\BZ_2\subset G_{\tdet_n}$ with
generator $\t\in GL(W)$.
It sends $e_i\ot f_j$ to $e_j\ot f_i$, so acting on $W$ it has $+1$ eigenspace $e_i\ot f_j+e_j\ot f_i$ for $i\leq j$ and $-1$ eigenspace $e_i\ot f_j-e_j\ot f_i$
for $1\leq i<j \leq n$. Thus it acts on the one-dimensional vector space
$(\La{n^2}W)^{\ot s}$ by  $((-1)^{\binom n2})^s$, i.e., by $-1$ if $n\equiv 2,3\tmod 4$  and $s$ is odd and by $1$ otherwise. We conclude
that there is an invariant as asserted above.
(In the language of \S\ref{detoring}, $sk^{s^{n^2}}_{(sn)^n,(sn)^n}=1$ for all $s$ when
$\binom n2$ is even, and $sk^{s^{n^2}}_{(sn)^n,(sn)^n}=1$ for even $s$ when $\binom n2$ is odd and is zero for odd $s$.)
\end{proof}

\exerone{Write out the proof of the non-normality of $\Perm^n_n$.}

\exerone{Show the same method gives another proof that $Ch_n(W)$ is not normal.}

\exerone{Show that the proof of Theorem \ref{kumarnorthm} holds for any reductive group with a nontrivial center (one gets
a $\BZ^k$-grading of modules if the center is $k$-dimensional), in particular it holds for
$G=GL(A)\times GL(B)\times GL(C)$. Use this to show that
$\s_r(Seg(\BP A\times \BP B\times \BP C))$ is not normal when $\tdim A=\tdim B=\tdim C=r>2$.}

\section{Determinantal hypersurfaces}\label{zhanglee}

Classically, there was interest in determining which smooth hypersurfaces of degree $d$ were expressible as a $d\times d$ determinant.
The result in the first
nontrivial case shows how daunting GCT might be.

\begin{theorem}[Letao Zhang and Zhiyuan Li]\label{zlthm}
The variety $\BP \{ P\in S^4\BC^4 \mid [P]\in \Det_4\}\subset \BP S^4\BC^4$  is   a hypersurface   of degree $640,224$.
\end{theorem}

The following \lq\lq folklore\rq\rq\ theorem was made explicit in  \cite[Cor. 1.12]{MR1786479}:
\begin{theorem}\label{bdet} Let $U=\BC^{n+1}$,  let $P\in S^dU$, and let  $Z=Z(P)\subset \BC\BP^n$ be the corresponding   hypersurface of degree $d$.
Assume $Z$ is smooth and choose any inclusion $U\subset \BC^{d^2}$. 

If $P\in   \tend(\BC^{d^2})\cdot [\tdet_d]$, we may form a  map between vector bundles
$M:\cO_{\pp n}(-1)^d\ra \cO^d_{\pp n}$ whose cokernel is a line bundle $L\ra Z$ with the properties:

i) $H^i(Z,L(j))=0$ for $1\leq i\leq n-2$ and all $j\in \BZ$

ii) $H^0(X,L(-1))=H^{n-1}(X,L(j))=0$ 

Conversely, if there exists $L\ra Z$ satisfying properties i) and ii), then $Z$ is determinantal
via a map $M$ as above whose cokernel is $L$.
\end{theorem}

If we are concerned with the hypersurface being in $\Det_n$, the first case where this is not automatic is for quartic surfaces, where
it is a codimension one condition:

\begin{proposition}\cite[Cor. 6.6]{MR1786479}\label{beauprop}
A smooth quartic surface is determinantal if and only if  it contains a nonhyperelliptic curve of genus $3$ embedded in $\pp 3$ by
a linear system of degree $6$.
\end{proposition}

 \begin{proof}[Proof of \ref{zlthm}]
From Proposition  \ref{beauprop}, 
the hypersurface  is  the locus of quartic surfaces
containing a (Brill-Noether general) genus $3$ curve $C$ of degree six.  This translates into the existence of a lattice polarization
$$\begin{matrix}
       & h  &  C\\
   h    &4  &  6\\
    C  &  6 &   4
\end{matrix}
$$
of discriminant $-(4^2-6^2)=20$.  By the Torelli theorems, the $K3$ surfaces with such
 a lattice polarization have codimension one in the moduli space of quartic $K3$ surfaces.

Let $D_{3,6}$ denote the locus of quartic surfaces containing a genus $3$ curve $C$ of degree six in $\BP^{34}=\BP(S^4\BC^4)$.
It  corresponds to the Noether-Lefschetz divisor $NL_{20}$ in the moduli space of the degree four $K3$ surfaces.
Here  $NL_d$ denotes the Noether-Lefschetz divisor, parameterizing the degree
 $4$ $K3$ surfaces whose Picard lattice has a rank $2$ sub-lattice containing $h$ with discriminant
 $-d$. (h is the polarization of the degree four  $K3$ surface, $h^2=4$.)

The Noether-Lefschetz number $n_{20}$, which is defined by the intersection number of 
$NL_{20}$ and a line in the moduli space of degree four  $K3$ surfaces, equals   the degree 
of $D_{3,6}$ in $\BP^{34}=\BP(S^4\BC^4)$.

The key fact is that   $n_d$ can be computed via the modularity 
of the generating series for any integer $d$. More precisely, the generating series
                         $F(q):=\sum_d n_d q^{d/8}$
is a modular form of level $8$, and can be expressed by a polynomial of $A(q)=\sum_n q^{n^2/8}$ and $B(q)=\sum_n (-1)^nq^{n^2/8}$.

The explicit expression of $F(q)$ is in \cite[Thm 2]{MaulPand}. 
As an application, the Noether-Lefschetz number $n_{20}$ is the coefficient of the term $q^{20/8}=q^{5/2}$, which is $640,224$.
\end{proof}

\section{Classical linear algebra and GCT}\label{linalgdetsect}

One potential source of new equations  for $\Det_n$ is to exploit classical identities the determinant satisfies.
What follows are ideas in this direction. This section   is joint unpublished work with
L. Manivel and N. Ressayre.

\subsection{Cayley's identity}
Let $\BC^{n^2}$ have coordinates $x^i_j$ and the dual space coordinates $y^i_j$. The classical Cayley identity (apparently first
due to 
Vivanti, see  \cite{MR3032306}) is
$$
\langle (\tdet_n(y)),(\tdet_n(x))^{s+1}\rangle  =\frac{(s+n)!}{s!}(\tdet_n(x))^{s}
$$
which may be thought of as a pairing between homogeneous polynomials of degree $n(s+1)$ and homogeneous differential operators
of order $n$ (compare with Proposition \ref{kconjequiv}). This and more general Bernstein-Sato type identities (again, see \cite{MR3032306}) appear as if they could be used to
obtain equations for $\Det_n$. So far we have only found rational equations in this manner. 

In more detail, \lq\lq$\tdet_n(y)$\rq\rq\ depends on the choice of identification of $\BC^{n^2}$
with $\BC^{n^2*}$ given by the coordinates, but one could, e.g. ask for polynomials $P \in S^nW$ such that
there exists some $Q\in S^nW^*$, with $G_P$ and $G_Q$ isomorphic and $\langle Q,P^{s+1}\rangle= \frac{(s+n)!}{s!}P^{s}$. 
 
\subsection{A generalization  of the Sylvester-Franke Theorem}

Let
$f: V\ra V$ be a diagonalizable linear map with distinct eigenvalues $\l_1\hd \l_{\bv}$.
The induced linear map
$f^{\ww k}: \La k V\ra \La k V$ has eigenvalues 
$\l_{i_1}\cdots \l_{i_k}$, $1\leq i_1<\cdots <i_k\leq \bv$. 
In particular $f^{\ww \bv}:\La\bv V\ra \La\bv V$ is multiplication by the scalar
$\tdet(f)=\l_1\cdots  \l_{\bv}$.
Now consider $\La k V$ as a vector space (ignoring its extra structure), and
$$[f^{\ww k}]^{\ww s}: \La s(\La k V)\ra \La s(\La k V).
$$
Let   
\begin{align}
cp_s : V\ot V^*& \ra \BC\\
\nonumber f&\mapsto \ttrace(f^{\ww s}),
\end{align}
denote   the $s$-th coefficient of
the characteristic polynomial.
We may consider $cp_s=Id_{\La s V}\in \La s V\ot \La s V^*\subset   S^s(V\ot V^*)$.
Recall that  $cp_{\bv}=\tdet$.

\begin{proposition}\label{sfgen} The degree $\bv p$ polynomial on $V\ot V^*$ given by
$f\mapsto (\tdet)^p(f)$ divides  the degree ${(\binom{\bv-1}k+p)k}$ polynomial
$f\mapsto cp_{\binom{\bv-1}k+p}(  f^{\ww k})$. 

In other words, for a $\bv \times \bv$ matrix $A$ with indeterminate
entries,  the degree $\bv p$ polynomial 
   $\tdet(A)^p$ divides the trace of the 
  $[ {\binom{\bv-1}k+p}]$-th  companion matrix of
  the $k$-th companion matrix of $A$.  
\end{proposition}

The Sylvester-Franke theorem is the special case $p=\binom{\bv-1}{k-1}$.

\begin{proof}
Assume $f$ has $\bv$ distinct eigenvalues.
The eigenvalues of $[f^{\ww k}]^{\ww s}$ are sums of terms of the 
form $\s_{J_1}\cdots \s_{J_s}$ where  
$\s_{J_m}=\l_{j_{m,1}}\cdots \l_{j_{m,k}}$ and the $\l_{j_{m,1}}\hd \l_{j_{m,k}}$ are distinct eigenvalues of $f$. 
Once every  $\l_j$ appears in a monomial to a power $p$, $\tdet^p$ divides the monomial.
  The result now
follows for linear maps with distinct eigenvalues by the pigeonhole principle.
Since the
subset of linear maps
with distinct eigenvalues forms a Zariski opens subset of $V\ot V^*$,
the 
 equality of polynomials  
holds everywhere. 
\end{proof}

\subsection{A variant of Proposition \ref{sfgen} for the Hessian}\label{almostgoodsect}
 
Say $g: \La 2 V^*\ra \La 2 V^*$ is a linear map
such that 
there exists a basis $v_1\hd v_{\bv}$ of $V$ with dual basis $\a^1\hd \a^{\bv}$
such that 
$$
g=\sum_{i<j}\l_{ij}\a^i\ww \a^j \ot v_i\ww v_j, 
$$
so $\l_{ij}$ are the eigenvalues of $g$. 
We will be concerned with the case  $g=f^{\ww(\bv-2)}$, where $f: V\ra V$ is a linear
map with   distinct eigenvalues $\l_1\hd \l_{\bv}$,   $v_1\hd v_{\bv}$ is an eigenbasis of $V$
with dual basis $\a^1\hd \a^{\bv}$, so  
$f=\l_1\a^1\ot v_1+\cdots + \l_{\bv} \a^{\bv}\ot v_{\bv}$. 
Then $ \l_{ij}=\l_1\cdots\l_{i-1}\l_{i+1}\cdots \l_{j-1}\l_{j+1}\cdots \l_{\bv}$.  

Consider the inclusion $in: \La 2 V^*\ot \La 2 V\subset S^2(V\ot V^*)$.
On decomposable elements it is given  by
$$
\a\ww\b\ot v\ww w \mapsto
(\a\ot v)\ot (\b\ot w) - (\a\ot w)\ot (\b\ot v)-(\b\ot v)\ot (\a\ot w)
+(\b\ot w)\ot (\a\ot v)
$$
The space  $V\ot V^*$ is self-dual  as a $GL(V)$-module,  with the natural  quadratic form 
$Q(\a\ot v)=\a(v)$, so we may
identify $S^2(V\ot V^*)$ as a subspace of  
$\tEnd (V\ot V^*)$ via the  linear map $Q^\flat: V^*\ot V \ra V\ot V^*$ given by
$\a^i\ot v_j\mapsto v_i\ot \a^j$. 

Say we have a map $g$ as above.  
Consider
$g^\flat:=Q^\flat \circ in(g): V\ot V^*\ra V\ot V^*$, then
$$
g^\flat=\sum_{i<j} \l_{ij}
[ 
(  v_i\ot \a^i)\ot (\a^j\ot v_j ) - (v_i\ot \a^j )\ot (\a^j\ot v_i)-(v_j\ot \a^i)\ot (\a^i\ot v_j )
+(v_j\ot \a^j )\ot (\a^i\ot v_i) 
],
$$
so, 
\begin{align*}
g^\flat (v_i\ot \a^j)& =  -\l_{ij}v_j\ot \a^i \  i\neq j,\\
g^\flat (v_i\ot \a^i)&= \sum_{j\neq i} \l_{ij} v_j\ot \a^j.
\end{align*}
Thus 
$g^\flat$ may be thought of as a sum of two linear maps, one
preserving the subspace $D:=\langle v_1\ot \a^1\hd
v^{\bv}\ot \a^{\bv}\rangle $ and another
preserving the subspace $D^c:= \langle v_i\ot \a^j \mid i\neq j\rangle$.

The $2\binom \bv 2$ eigenvalues of $g^\flat|_{D^c}$ are $\pm\l_{ij}$.
Write $\psi_s$ for the coefficients of the characteristic polynomial of
$g^\flat|_{D^c}$. Since the eigenvalues
come paired with their negatives,   $\psi_s=0$ when $s$ is odd.

With respect to the given basis, the matrix for
$g^\flat|_D$ is a symmetric matrix with zeros on the diagonal, whose
off diagonal entries are the $\l_{ij}$. Write the coefficients of
the characteristic polynomial of  $g^\flat|_D$  as $\z_1\hd \z_{\bv}$,  and note
that $\z_1=0$, $\z_2=\sum_{i<j}  \l_{ij}^2$, $\z_3= 2\sum_{i<j<k}\l_{ij}\l_{ik}\l_{jk}$.

Now let $g=f^{\ww (\bv-2)}$ as above and we compare the determinant of $f$ with
the coefficients of the characteristic polynomial of the Hessian  $H(\tdet(f))$.
(Invariantly, $\tdet(f)=f^{\ww n}$ and $H: S^n(V\ot V^*)\ra S^2(V\ot V^*)\ot
S^{n-2}(V\ot V^*)$ is the $(2,n)$-polarization, so
$H(\tdet(f))=in(f^{\ww 2})\ot in(f^{\ww n-2})$.)

Observe that   $\tdet(f)^{2(s+1-\bv)}$ divides $\psi_{2s}$   and 
$\tdet(f)^k$ divides   $\z_{k+2}$. 
Also note that $\z_{\bv-1}=2Q\tdet_{\bv}^{\bv-2}$,  $\z_{\bv}=(\bv-1)(\tdet_{\bv})^{\bv-2}$,
and $\psi_{\bv^2-\bv}=(-1)^{\binom \bv 2}(\tdet_{\bv})^{(\bv-1)(\bv-2)}$.
 
Recall
$cp_j(A_1+A_2)=\sum_{\a=0}^j cp_{\a}(A_1)cp_{j-\a}(A_2)$.
Thus
\begin{align*}
cp_{2k}(H(det(f))&= \z_{2k}+\z_{2k-2} \psi_2 +\z_{2k-4} \psi_4+\cdots +\z_2\psi_{2k-2}+\psi_{2k}, \\
cp_{2k+1}(H(det(f))&= \z_{2k+1}+\z_{2k-1} \psi_2 +\z_{2k-3} \psi_4+\cdots +\z_3\psi_{2k-2}.
\end{align*}

We conclude:

\begin{theorem}\label{sfturbo} Let $Q\in S^2(V\ot V^*)$ be the canonical contraction,
so $S^2(V\ot V^*)\subset \tEnd(V\ot V^*)$.  Write
$CP(H(\tdet_{\bv}))=\sum cp_{\bv^2-j}y^j$ for the characteristic polynomial.
Then
\begin{align*}
cp_{0}&=1\\
cp_{ 1}&=0\\
cp_{ 3}&=\tdet_{\bv}R_{2{\bv} -6}\\
cp_{ 5}&=\tdet_{\bv}R_{4{\bv} -10}\\
&\vdots\\
cp_{ 2k  }&=\tdet_{\bv}^{2(s-\bv +1)} R_{2(\bv^2-2s-\bv)}\ k>\bv \\
cp_{ 2k+1 }&=\tdet_{\bv}^{2(s-\bv)+1} R_{2(\bv^2-2s-1)}\ k>\bv \\
cp_{\bv^2-1}&=  2(\tdet_{\bv})^{{\bv}({\bv}-2)-1}Q \\
cp_{\bv^2}&= (-1)^{\binom{{\bv}+1}2}({\bv}-1)(\tdet_{\bv})^{{\bv}({\bv}-2)}
\end{align*}
where $R_k$ is a polynomial of degree $k$. Moreover $\det_{\bv}$ does not divide the even
$cp_s$ for $s<2\bv+1$. 
\end{theorem}

\begin{remark} The equality   $cp_{\bv^2}=(-1)^{\binom{{\bv}+1}2}({\bv}-1)(\tdet_{\bv})^{{\bv}({\bv}-2)}$ is due to
B. Segre.
\end{remark}

\exerone{Prove the analog of the B. Segre equality for the discriminant $\Delta \in S^4(S^3\BC^2)$ (the equation
of the dual variety of $v_3(\pp 1)\dual$).
Namely, if one takes $\Delta=27x_1^2x_4^2+4x_1x_3^3+4x_2^3x_4-x_2^2x_3^2-18x_1x_2x_3x_4$, then
$\tdet(H(\Delta))=3888\Delta^2$.
}

\begin{problem} Find all the components of $\Dual_{4,4,1}$, show $\ol{GL_{4}\cdot \Delta}$
is an irreducible component of $\Dual_{4,4,1}$,  and find defining equations for that
component.
\end{problem}

\subsection{A cousin of $\Det_n$}
In   GCT   one is interested in orbit closures $\ol{GL(W)\cdot [P]}\subset S^dW$
where $P\in S^dW$. 
One cannot make sense of the coefficients of the characteristic polynomial of
$H(P)\in S^2W\ot S^{d-2}W$ without choosing an isomorphism $Q:W\ra W^*$.

If $P=\tdet_n$ and we choose bases to express elements of $W$ as
$n\times n$ matrices, then taking $Q(A)=\ttrace(AA^T)$ will give the desired
identification to enable us to potentially use the equations 
implied by Theorem \ref{sfturbo}. (Note that taking $Q'(A)=\ttrace(A^2)$ will not.)
However these are equations for $\ol{O(W,Q)\cdot \tdet_n}$ rather than $\Det_n$.

The proof of  Theorem  \ref{sfturbo} used the fact that a Zariski open subset of the space
of matrices is diagonalizable under the action of $GL(V)$ by conjugation.
We no longer have this action, but instead, writing $W=E\ot F$, we
have the intersection of the stabilizers of $\tdet_n$ and $Q$, 
i.e.,  $O(W,Q)\cap [(SL(E)\times SL(F))/\mu_n\rtimes \BZ_2]$.

\begin{proposition}
The connected component of
the identity of  $O(W,Q)\cap [SL(E)\times SL(F)\rtimes \BZ_2]$ is  $SO(E)\times SO(F)$.
\end{proposition}
\begin{proof}
The inclusion $SO(E)\times SO(F)\subseteq O(W,Q)\cap [SL(E)\times SL(F)\rtimes \BZ_2]$
is clear. To see the other inclusion, note that over $\BR$, $SO(n,\BR)\times SO(n,\BR)$ is a maximal
compact subgroup of $SL(n,\BR)\times  SL(n,\BR)$. The equations for the Lie algebra of the stabilizer
are linear, and the rank of a linear system of equations is the same over $\BR$ or $\BC$,
so the result holds over $\BC$.
\end{proof}

\begin{proposition}The $SO(E)\times SO(F)$ orbit of the diagonal matrices contains  a Zariski open subset
of $E\ot F$.
\end{proposition}
\begin{proof} We show the kernel of the differential of the map
$SO(E)\times SO(F)\times D\ra E\ot F$ at $(Id_E,Id_F,\d)$ is zero,
where $\d$ is a sufficiently general diagonal matrix.
The differential   is $(X,Y,\d')\mapsto \d' + X\d + \d Y$, where $\d'$ is  
diagonal.
The matrix $X\d+\d Y$ has zeros on the diagonal and its $(i,j)$-th
entry is
$X^i_j\d_j+\d_iY^j_i$. Write out the $2\binom n2$ matrix in the $\d_i$ for
the $2\binom n2$ unknowns $X,Y$ resulting from  the equations  $X^i_j\d_j+\d_iY^j_i=0$.
Its determinant is $\Pi_{i<j}(\d_i^2-\d_j^2)$, which is nonzero as long as the  $\d_j^2$ are distinct.
\end{proof}

We apply Theorem \ref{sfturbo} to obtain:

\begin{theorem}\label{cpeqns} Let $P\in  \ol{O(W,Q)\cdot [\tdet_n]}$, then
$P$ divides   $\ttrace (H(P)^{\ww {j}}) \in S^{{j}(n-2)}W$ for  each odd $j>1$ up to $j=2n+1$.
In particular we obtain modules of equations of degrees $(j-1)(d-1)$ for
$\ol{O(W,Q)\cdot [\tdet_n]}$ for $j$ in this range. 
\end{theorem}

\subsection{Relation to GCT?}
Since $\tdim O(W,Q)$ is roughly half that of $GL(W)$, and it contains
a copy of $GL_{\lfloor \frac {\bw ^2}2\rfloor}$, e.g., if $\bw=2n$ is even and $Q=x^1y^1+\cdots + x^n y^n$, then
$$
\left\{ \begin{pmatrix} A& 0 \\ 0& A\inv\end{pmatrix}\mid A\in GL_n\right\} \subset O(W,Q), 
$$
one might  hope to use the variety 
$\ol{O(W,Q)\cdot [\tdet_n]}$ as a substitute for $\Det_n$ in the GCT program, since
we have many equations for it, and these equations do not vanish identically on cones.

Consider $P\in S^m\BC^M$ and $\ell^{n-m}P\in S^n\BC^{M+1}\subset S^n\BC^N=S^nW$. Taking 
the na\"\i ve coordinate embedding such that $Q$ restricted to  $\BC^{M+1}$ is nondegenerate gives:
\begin{align*}&
\ttrace(H_N(\ell^{n-m}P)^{\ww 3})=\\
&
\ell^{3(n-m)}\ttrace(H_M(P)^{\ww 3})+ \ell^{3(n-m)-2}[P\ttrace(H_M(P)^{\ww 2})
+ \sum_{i<j}(2P_iP_jP_{ij}-P_i^2P_{jj}-P_j^2P_{ii})]
\end{align*}
where $P_i=\frac{\partial P}{\partial x_i}$ etc...
When does $\ell^{n-m}P$ divide this expression?
We need that $P$ divides $\ttrace(H_M(P)^{\ww 3})$ and
$\sum_{i<j}(2P_iP_jP_{ij}-P_i^2P_{jj}-P_j^2P_{ii})$.  But these conditions are independent of $n,N$ so there is 
no hope of getting this condition asymptotically.
However, taking a more complicated inclusion might erase this problem.


\section{Appendix: Complexity theory}\label{complexapp}
  In a letter to von Neumann (see \cite[Appendix]{Sipser})  G\"odel 
  tried to quantify what we mean by \lq\lq intuition\rq\rq, or more precisely  the apparent difference between intuition and systematic problem solving.
At the same time, researchers in the Soviet Union were trying to determine if \lq\lq brute force search\rq\rq\ was avoidable in
solving problems such as the   traveling salesman problem,  where there seems to be no fast way to find a solution,
but a proposed solution can be easily checked. (If I say I have found a way to visit twenty cities by traveling less than
a thousand miles, you just need to look at my plan and check the distances.)   These discussions eventually gave rise to the complexity
classes  $\p$, which models problems admitting a fast algorithm to produce a solution, and $\np$ which models problems admitting a fast algorithm
to verify a proposed solution.

The \lq\lq problems\rq\rq\   relevant to us are sequences of polynomials or multi-linear maps (i.e. tensors), and the goal  is to
find lower bounds on the complexity of evaluating them, or otherwise  to find efficient algorithms to do so. Geometry
has so far been more useful in determining lower bounds.

\subsection{Arithmetic circuits and complexity classes}
Recall the definitions regarding circuits from Definition \ref{arithcirdef}.

\begin{definition} A circuit $\cC$ is {\it weakly skew} if for each multiplication gate $v$, receiving the outputs
of gates $u,w$, one of $\cC_u$, $\cC_w$ is disjoint from the rest of $\cC$. (I.e., the only output of, say $\cC_u$, 
is the edge entering $v$.) A circuit is {\it multiplicatively disjoint} if, for every multiplication gate $v$
receiving the outputs
of gates $u,w$, the subcircuits $\cC_u$, $\cC_w$ do not intersect.
\end{definition}

 \begin{figure}[!htb]\begin{center}
\includegraphics[scale=.7]{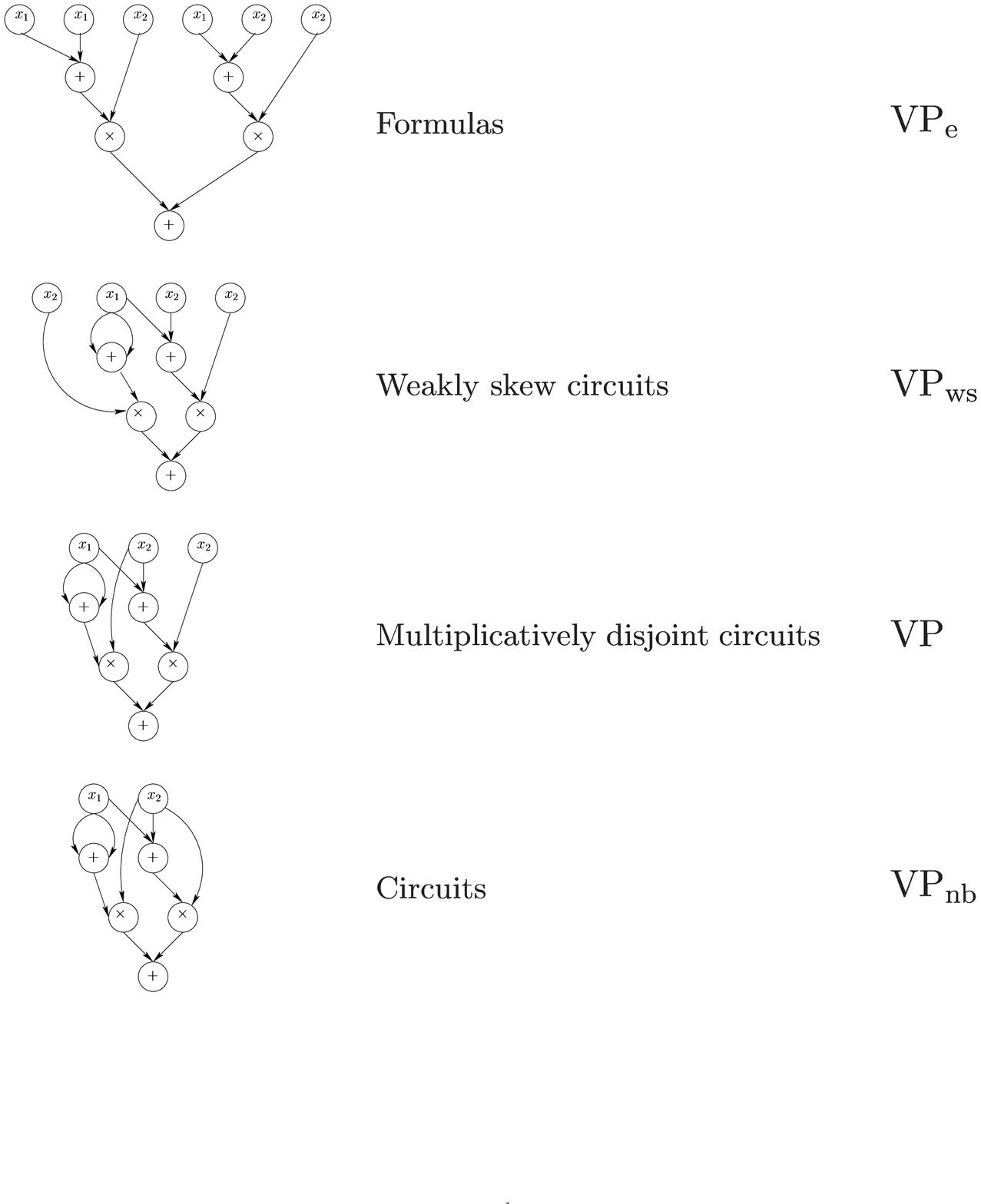}
\end{center}
\end{figure}

\begin{definition} Let $(f_n)$ be a sequence of polynomials. We say
\begin{itemize}
\item $(f_n)\in \vp_e$  if there exists a sequence of formulas  $\cC_n$ of polynomial size calculating $f_n$.

\item $(f_n)\in \vp_{ws}$  if there exists a sequence of weakly skew  circuits $\cC_n$ of polynomial size calculating $f_n$.

\item $(f_n)\in \vp$  if there exists a sequence of multiplicatively disjoint  circuits $\cC_n$ of polynomial size calculating $f_n$.

\item $(f_n)\in \vp_{nb}$  if there exists a sequence of    circuits $\cC_n$ of polynomial size calculating $f_n$.
\end{itemize}
\end{definition}

These definitions agree with the standard ones, see \cite{mapo:04}. In particular, for the first three, 
they require $\tdeg(f_n)$ to grow like a polynomial in $n$.  
The class $\vnp$ has a more complicated definition: $(f_n)$ is defined to be in $\vnp$ if  
there exists a polynomial $p$ and a sequence $(g_n)\in \vp$ such that 
$$
f_n(x)=\sum_{\ep\in \{0,1\}^{p(|x|)}}g_n(x,\ep).
$$

Valiant's conjectures are:

\begin{conjecture}[Valiant]\cite{vali:79-3} $\vp\neq \vnp$, that is, there does not exist a polynomial size
circuit computing the permanent.
\end{conjecture}

\begin{conjecture}[Valiant]\cite{vali:79-3} $\vp_{ws}\neq\vnp$, that is $dc(\tperm_m)$ grows faster than
any polynomial.
\end{conjecture}

\subsection{Complete problems}
The reason complexity theorists love the permanent so much is that it counts the number of
perfect matchings of a bipartite graph, a central counting problem in combinatorics. It is {\it complete} for the class $\vnp$. (A sequence
is complete for a class if it belongs to the class and any other sequence in the class
can be reduced to it at the price of a polynomial increase in size.)

The sequence of polynomials
given by iterated matrix multiplication
of $3\times 3$ matrices,    $IMM^n_3\in S^n(\BC^{9 n})$ 
where $IMM^n_3(X_1\hd X_n)=\ttrace(X_1\cdots X_n)$ is complete for $\vp_e$, see 
\cite{clevebenor}. 

The complexity class $\vp_{ws}$ is not natural from the perspective of complexity theory. It exists only because
the sequence $(\tdet_n)$ is $\vp_{ws}$-complete, however, there exists a more natural (from the
perspective of complexity theory) class, called $\vqp$ (see, e.g, \cite[\S 21.5]{BCS}) for which it is also complete.
  
\bibliographystyle{amsplain}
 
\bibliography{Lmatrix}

\end{document}